\documentclass[12pt, reqno, a4paper, oneside]{amsart}

\usepackage{amsmath,amssymb,amsthm,eucal}
\usepackage[a4paper]{geometry}
\usepackage[dvipdfm,colorlinks=true]{hyperref}
\usepackage[all]{xy}
\usepackage{mathtools,here}

\newtheorem{thm}{Theorem}[section]
\newtheorem{prop}[thm]{Proposition}
\newtheorem{lem}[thm]{Lemma}
\newtheorem{cor}[thm]{Corollary}

\newtheorem{abcthm}{Theorem}
\newtheorem{abccor}[abcthm]{Corollary}

\theoremstyle{definition}

\newtheorem{exam}[thm]{Example}

\newtheorem{quest}[thm]{Question}

\theoremstyle{remark}
\newtheorem{rem}[thm]{Remark}

\newcommand{\Z}{\mathbb{Z}}
\newcommand{\Q}{\mathbb{Q}}
\newcommand{\R}{\mathbb{R}}

\newcommand{\F}{\mathbb{F}}
\renewcommand{\P}{\mathcal{P}}
\renewcommand{\H}{\mathcal{H}}
\renewcommand{\S}{\mathcal{S}}

\DeclareMathOperator{\Hom}{Hom}
\DeclareMathOperator{\tr}{tr}
\DeclareMathOperator{\sgn}{sgn}
\DeclareMathOperator{\Emb}{Emb}

\begin{document}

  \title{Spaces of commuting elements in the classical groups}
    \author{Daisuke Kishimoto}
    \address{Department of Mathematics, Kyoto University, Kyoto, 606-8502, Japan}
    \email{kishi@math.kyoto-u.ac.jp}

    \author{Masahiro Takeda}
    \address{Department of Mathematics, Kyoto University, Kyoto, 606-8502, Japan}
    \email{m.takeda@math.kyoto-u.ac.jp}


    \begin{abstract}
      Let $G$ be the classical group, and let $\Hom(\Z^m,G)$ denote the space of commuting $m$-tuples in $G$. First, we refine the formula for the Poincar\'e series of $\Hom(\Z^m,G)$ due to Ramras and Stafa by assigning (signed) integer partitions to (signed) permutations. Using the refined formula, we determine the top term of the Poincar\'e series, and apply it to prove the dependence of the topology of $\Hom(\Z^m,G)$ on the parity of $m$ and the rational hyperbolicity of $\Hom(\Z^m,G)$ for $m\ge 2$. Next, we give a minimal generating set of the cohomology of $\Hom(\Z^m,G)$ and determine the cohomology in low dimensions. We apply these results to prove homological stability for $\Hom(\Z^m,G)$ with the best possible stable range. Baird proved that the cohomology of $\Hom(\Z^m,G)$ is identified with a certain ring of invariants of the Weyl group of $G$, and our approach is a direct calculation of this ring of invariants.
    \end{abstract}

  \maketitle

\setcounter{tocdepth}{1}
\tableofcontents


\section{Introduction}\label{Introduction}

Let $G$ be a compact connected Lie group. The \emph{space of commuting elements} in $G$, denoted by $\Hom(\Z^m,G)$, is the subspace of the Cartesian product $G^m$ consisting of $(g_1,\ldots,g_m)\in G^m$ such that $g_1,\ldots,g_m$ are pairwise commutative. The topology of $\Hom(\Z^m,G)$ has been studied intensely in recent years; in particular, its homological and homotopical features have been studied by a number of authors \cite{AC,ACG,B,BJS,CS2,C,GPS,RS1,RS2,STG}. On the other hand, as in \cite{BFM}, $\Hom(\Z^m,G)$ is identified with the based moduli space of flat $G$-bundles over an $m$-torus. Then it has been studied also in the context of geometry and physics, including work of Kac and Smilga \cite{KS} and Witten \cite{W1,W2} on supersymmetric Yang-Mills theory. We refer the reader to a comprehensive survey \cite{CS1} for the basics of $\Hom(\Z^m,G)$ and a list of related work.

The purpose of this paper is two-fold. Let $G$ be the classical group. First, we refine the formula for the Poincar\'e series of $\Hom(\Z^m,G)$ due to Ramras and Stafa \cite{RS1} in terms of integer partitions. Using the refined formula, we determine the top term of the Poincar\'e series, which applies to prove the dependence of the topology of $\Hom(\Z^m,G)$ and the rational hyperbolicity of $\Hom(\Z^m,G)$. Second, based on the description of the cohomology of $\Hom(\Z^m,G)$ due to Baird \cite{B} in terms of a certain ring of invariants of the Weyl group of $G$, we give a minimal generating set of the cohomology and determine the cohomology in low dimensions. These results apply to prove homological stability for $\Hom(\Z^m,G)$ with the best possible stable range.


\subsection{Poincar\'e series}

Let $\Hom(\Z^m,G)_1$ denote the path-component of $\Hom(\Z^m,G)$ containing $(1,\ldots,1)\in G^m$, and let $W$ denote the Weyl group of $G$. In \cite{RS1}, Ramras and Stafa gave a formula for the Poincar\'e series of $\Hom(\Z^m,G)_1$:
\begin{equation}
  \label{Ramras-Stafa intro}
  P(\Hom(\Z^m,G)_1;t)=\frac{1}{|W|}\prod_{i=1}^r(1-t^{2d_i})\sum_{w\in W}\frac{\det(1+tw)^m}{\det(1-t^2w)}
\end{equation}
Remarks on this formula are in order. The Poincar\'e series of $\Hom(\Z^m,G)_1$ means that of the cohomology of $\Hom(\Z^m,G)_1$ over a field of characteristic zero or prime to the order of $W$. Then the Poincar\'e series of $\Hom(\Z^m,G)_1$ does not depend on the coefficient field unless it is of characteristic zero or prime to $|W|$. The integers $d_1,\ldots,d_r$ are the characteristic degrees of the Weyl group $W$, and the determinants are taken by using the reflection group structure of $W$.

Although the formula \eqref{Ramras-Stafa intro} is beautiful, it is less computable. For example, we cannot determine the degree of the Poincar\'e series directly from the formula \eqref{Ramras-Stafa intro}. Then in order to get information from the Poincar\'e series, we must refine the formula into a more computable form. Suppose $G$ is the classical group. Then the Weyl group $W$ is a symmetric group, a signed symmetric group or its subgroup. So we can assign (signed) integer partitions to elements of $W$ via the (signed) cyclic decomposition of (signed) permutations. This enables us to refine the formula \eqref{Ramras-Stafa intro}, and we give it for unitary groups here. The formulae for other classical groups will be given in Section \ref{Poincare series}.

\begin{abcthm}\label{Poincare intro}
  For $\lambda=(\lambda_1^{n_1},\ldots,\lambda_l^{n_l})\vdash k\le n$, let
  \[
    p^{m,n}_\lambda(t)=\frac{t^{(m-2)(n-k)}}{\lambda_1\cdots\lambda_ln_1!\cdots n_l!}\prod_{i=1}^l\left((-1)^{m(\lambda_i-1)}t^{(m-2)\lambda_i}+\frac{(1+(-1)^{\lambda_i+1}t^{\lambda_i})^m}{1-t^{2\lambda_i}}\right)^{n_i}
  \]
  where we set $q_\lambda^{m,n}(t)=t^{(m-2)n}$ if $\lambda$ is the empty partition. Then the Poincar\'e series of $\Hom(\Z^m,U(n))_1$ is given by
  \[
  P(\Hom(\Z^m,U(n))_1;t)=
  \begin{cases}
    \displaystyle\prod_{i=1}^n(1-t^{2i})\sum_{k=n-1}^n\sum_{\lambda\vdash k}(-1)^{n+k}p_\lambda^{m,n}(t)&(m\text{ even})\\
    \displaystyle\prod_{i=1}^n(1-t^{2i})\sum_{k=0}^n\sum_{\lambda\vdash k}(-1)^kp_\lambda^{m,n}(t)&(m\text{ odd}).
  \end{cases}
  \]
\end{abcthm}


\subsection{Top terms}

Using the formula  in Theorem \ref{Poincare intro} together with enumeration of (signed) integer partitions, we can determine the top term of the Poincar\'e series of $\Hom(\Z^m,G)$ when $G$ is the classical group.

\begin{abcthm}\label{top term intro}
  The top term of the Poincar\'e series of $\Hom(\Z^m,U(n))_1$ is $t^{n^2+(m-1)n}$ for $m$ odd and $\binom{m+n-2}{m-1}t^{n^2+(m-2)n+1}$ for $m$ even.
\end{abcthm}

The top terms for other classical groups will be also determined in Section \ref{Top terms}. Clearly, the formula in Theorem \ref{Poincare intro} depends on the parity of $m$, and Theorem \ref{top term intro} shows that the Poincar\'e series themselves depend on the parity of $m$. This dependence can be stated in terms of palindromicity.

\begin{abccor}\label{palindromic intro}
  Let $G$ be the classical group which is neither trivial nor $S^1$. Then the Poincar\'e series of $\Hom(\Z^m,G)_1$ is palindromic if and only if $m$ is odd.
\end{abccor}

As mentioned above, our refinement of the formula \eqref{Ramras-Stafa intro} for the classical groups is based on the assignment of (signed) integer partitions to (signed) permutations. Then it does not apply to the exceptional Lie groups. However, we can calculate the Poincar\'e series of $\Hom(\Z^2,G)_1$ for $G$ exceptional directly from the formula \eqref{Ramras-Stafa intro} of Ramras and Stafa by using a computer. We will give the result in Appendix. In particular, we determine the top term of the Poincar\'e series of $\Hom(\Z^2,G)_1$ for every simple Lie group $G$. Using these results, we can determine the Poincar\'e series of $\Hom(\Z^2,G)_1$ for every compact connected Lie group $G$.

\begin{abcthm}\label{top term general intro}
  Let $G$ be a compact connected Lie group with simple factors $G_1,\ldots,G_k$. Then the top term of the Poincar\'e series of $\Hom(\Z^2,G)_1$ is
  \[
    (\mathrm{rank}\,G_1+1)\cdots(\mathrm{rank}\,G_k+1)t^{\dim G+\mathrm{rank}\,\pi_1(G)}.
  \]
\end{abcthm}

This result has an application to the rational homotopy of $\Hom(\Z^m,G)_1$. Let $X$ be a simply-connected finite complex. As in \cite[Part IV]{FHT}, it is well known that $\sum_{n\ge 1}\pi_n(X)\otimes\Q$ is either finite or of exponential growth. In the former case, $X$ is called \emph{rationally elliptic}, and the latter, \emph{rationally hyperbolic}.

\begin{abccor}\label{hyperbolic intro}
  Let $G$ be the non-trivial compact simply-connected Lie group. Then for $m\ge 2$, $\Hom(\Z^m,G)_1$ is rationally hyperbolic.
\end{abccor}


\subsection{Cohomology generators}

Let $T$ denote a maximal torus of $G$, and let $\F$ be a field of characteristic zero or prime to the order of $W$. In \cite{B}, Baird proved that there is an isomorphism
\begin{equation}\label{Baird intro}
  H^*(\Hom(\Z^m,G)_1;\F)\cong (H^*(G/T;\F)\otimes H^*(T;\F)^{\otimes m})^W.
\end{equation}
Since $H^*(G/T;\F)$ is identified with the ring of coinvariants of $W$, the RHS is completely determined by the Weyl group $W$. We calculate the invariant ring on the RHS to give a minimal generating set of the cohomology of $\Hom(\Z^m,G)_1$ over $\F$ for the classical group $G$, except for $SO(2n)$. Here we show the result for $U(n)$, and the results for other classical groups will be given in Section \ref{Cohomology generators}.

\begin{abcthm}\label{generator intro}
  The cohomology of $\Hom(\Z^m,U(n))_1$ over a field $\F$ of characteristic zero or prime to $n!$ is minimally generated by
  \[
    \S\coloneqq\{z(d,I)\mid 1\le d\le n\text{ and }\emptyset\ne I\subset\{1,2,\ldots,m\}\text{ such that }d+|I|-1\le n\}
  \]
  where $|z(d,I)|=2d+|I|-2$.
\end{abcthm}

Let $\widetilde{\H}\coloneqq H^*(BT;\F)\otimes H^*(T;\F)^{\otimes m}$ and $\H\coloneqq H^*(G/T;F)\otimes H^*(T;\F)^{\otimes m}$. Then $\H$ is a quotient of $\widetilde{\H}$, and by \eqref{Baird intro}, we aim to give a minimal generating set of $\H^W$. Our calculation of the invariant ring $\H^W$ is quite direct and consists of two parts. First, we define a subset, say $\overline{\S}$, of $\widetilde{\H}^W$ from polynomial invariants of the Weyl group $W$, and show that $\overline{\S}$ generate $\H^W$. The proof is done by an ordering on $\widetilde{\H}$, which is a key ingredient. Second, we choose a subset of $\overline{\S}$, say $\S$, minimally generating $\H^W$. We show that every element of $\overline{\S}$ is given in terms of $\S$ by using a polynomial tensor exterior algebra analog of the Newton formula for symmetric polynomials and power sums. Then we prove minimality of $\S$ by describing $\H^W$ in low dimension. For unitary groups, we prove the following, where analogous results for other classical groups will be given in Section \ref{Cohomology generators}. Let $\F\langle S\rangle$ denote a free graded commutative algebra generated by a graded set $S$.

\begin{abcthm}\label{low dim intro}
  Let $\F$ be a field of characteristic zero or prime to $n!$. Then the map
  \[
    \F\langle\S\rangle\to H^*(\Hom(\Z^m,U(n))_1;\F)
  \]
  is an isomorphism in dimension $\le 2n-m$, where $\S$ is as in Theorem \ref{generator intro}.
\end{abcthm}

Applying the result for $SU(n)$, we determine the cohomology of $\Hom(\Z^2,SU(n))_1$ for $n=2,3$. As is seen in this example calculation, relations among our generators are quite complicated, and we do not have a general scheme to get relations.


\subsection{Homological stability}

A sequence of spaces $X_1\to X_2\to X_3\to\cdots$ is said to satisfy \emph{homological stability} if for every $i\ge 0$, the induced sequence
\[
  H_i(X_1)\to H_i(X_2)\to H_i(X_3)\to\cdots
\]
eventually consists of isomorphisms. Clearly, each sequence of the classical groups in the same series satisfies homological stability, and the classical result of Nakaoka \cite{N} shows that the sequence of (the classifying spaces of) the symmetric groups satisfy homological stability. Besides these classical results, many important series of groups \cite{A,Ha,He,Q} and spaces \cite{CF,GRW,K} are proved to satisfy homological stability, and recently, the techniques for proving homological stability are rapidly developed \cite{CE,CF,GKRW}.

Let $G_1\to G_2\to G_3\to\cdots$ be one of the series of $U(n),SU(n),Sp(n),SO(2n+1)$. Then there is a sequence
\[
  \Hom(\Z^m,G_1)_1\to\Hom(\Z^m,G_2)_1\to\Hom(\Z^m,G_3)_1\to\cdots.
\]
By applying the technique of Church and Farb \cite{CF}, Ramras and Stafa \cite{RS2} proved that the sequence satisfies homological stability such that the map
\[
  H_i(\Hom(\Z^m,G_n)_1;\Q)\to H_i(\Hom(\Z^m,G_{n+1})_1;\Q)
\]
is an isomorphism for $i\le n-\lfloor\sqrt{n}\rfloor$. As an example application of Theorem \ref{Poincare intro}, we will calculate the Poincar\'e series of $\Hom(\Z^2,SU(n))_1$ for $2\le n\le 5$ in Example \ref{example Poincare series}. As far as looking at these Poincar\'e series, those of $\Hom(\Z^2,SU(n))_1$ and $\Hom(\Z^2,SU(n+1))_1$ coincide in degree $\le 2n-1$. This suggest possibility of extending the stable range for $U(n)$, i.e. those degrees for which stability holds.

By construction, the generating set $\S$ in Theorem \ref{generator intro} is natural with respect to the inclusion $U(n)\to U(n+1)$. Then by Theorem \ref{low dim intro} together with a bit of effort, we can give an alternative proof for homological stability of $\Hom(\Z^m,G_n)_1$, which provides the best possible stable range.

\begin{abcthm}\label{homology stability intro}
  For $n\ge m$, the map
  $$H_*(\Hom(\Z^m,U(n))_1;\Q)\to H_*(\Hom(\Z^m,U(n+1))_1;\Q)$$
  is an isomorphism for $*\le 2n-m+1$ and is not surjective for $*=2n-m+2$.
\end{abcthm}

Homological stability for other series of the classical groups will be given in Section \ref{Cohomology generators}.


\subsection{Outline}

We recall in Section \ref{Cohomology and Weyl groups} the map giving an isomorphism of Baird \eqref{Baird intro} above, where this isomorphism is the basis of our study. In Section \ref{Integer partitions and permutations}, we connect (signed) permutations to (signed) integer partitions, and show some enumerative properties of these connections. We begin Section \ref{Poincare series} with giving a short alternative proof for the formula\eqref{Ramras-Stafa intro} due to Ramras and Stafa. Then we compute the Poincar\'e series of $\Hom(\Z^m,G)_1$ for the classical group $G$ by applying the results in Section \ref{Integer partitions and permutations} to the formula \eqref{Ramras-Stafa intro}. We determine in Section \ref{Top terms} the top term of the Poincar\'e series and give its applications. In Section \ref{Cohomology generators}, we investigate the invariant ring of the Weyl group on the RHS of \eqref{Baird intro} to give its minimal generators and description in low dimensions. In Section \ref{Homological stability}, we apply the results in Section \ref{Cohomology generators} to show that the sequence $\Hom(\Z^m,G_n)_1$ satisfies homological stability and obtain the best possible stable range, where $G_n$ is as above. We end this paper in Section \ref{Questions} by posing several questions, arising in our study, connected to both topology and representation theory.


\subsection{Acknowledgement}

The first author was supported by JSPS KAKENHI No. 17K05248.


\section{Cohomology and Weyl groups}\label{Cohomology and Weyl groups}

In this section, we recall the result of Baird \cite{B} describing the cohomology of $\Hom(\Z^m,G)_1$ in terms of a ring of invariants of the Weyl group of $G$. This result is the basis of our study.

Throughout this section, let $G$ be a compact connected simple Lie group, let $T$ be its maximal torus, and let $W$ denote the Weyl group of $G$. Consider the action of $W$ on $G/T\times T^m$ given by
$$w\cdot(gT,t_1,\ldots,t_m)=(gwT,w^{-1}t_1w,\ldots,w^{-1}t_mw)$$
for $w\in W,g\in G,t_1,\ldots,t_m\in T$. Then the naive map
$$G\times T^m\to\Hom(\Z^m,G)_1,\quad(g,t_1,\ldots,t_m)\mapsto(gt_1g^{-1},\ldots,gt_mg^{-1})$$
for $g\in G,t_1,\ldots,t_m$ defines a map
$$\phi\colon G/T\times_WT^m\to \Hom(\Z^m,G)_1.$$
Let $\F$ be a field of characteristic zero or prime to the order of $W$. In \cite{B}, Baird proved that the map $\phi$ is an isomorphism in cohomology over a field $\F$. Since the action of $W$ on $G/T\times T^m$ is free, we have an isomorphism
$$H^*(G/T\times_WT^m;\F)\cong(H^*(G/T;\F)\otimes H^*(T;\F)^{\otimes m})^W.$$
Then we can restate the result of Baird as follows.

\begin{thm}\label{Baird}
  Let $\F$ be a field of characteristic zero or prime to the order of $W$. Then there is an isomorphism
  $$H^*(\Hom(\Z^m,G)_1;\F)\cong (H^*(G/T;\F)\otimes H^*(T;\F)^{\otimes m})^W.$$
\end{thm}

We recall the characteristic degrees of the Weyl group $W$. Suppose $G$ is of rank $n$, and let
$$\P(n)\coloneqq\F[x_1,\ldots,x_n],\quad|x_i|=2$$
where $\F$ is a field as above. Then $\P(n)$ is identified with the cohomology of the classifying space $BT$, so that $W$ acts on $\P(n)$. By the Shephard-Todd theorem, we have
\begin{equation}
  \label{polynomial invariant}
  \P(n)^W=\F[a_1,\ldots,a_n],\quad|a_i|=2d_i
\end{equation}
for some integers $d_1,\ldots,d_n$. We call these integers $d_1,\ldots,d_n$ the \emph{characteristic degrees} of $W$. Notice that the number of the characteristic degrees of $W$ coincides with the rank of $G$. We give a table including information on the Weyl groups of simple Lie groups, where $\Sigma_n$ and $B_n$ denote the symmetric group and the signed symmetric group, respectively, and $B_n^+$ denotes the subgroup of $B_n$ consisting of signed permutation with total sign one.

\renewcommand{\arraystretch}{1.3}

\begin{table}[H]
  \centering
  \begin{tabular}{p{1.25cm}p{2.3cm}p{1.3cm}p{1.2cm}p{2.25cm}l}
    \hline
    Type&Lie group&Rank&$W$&$|W|$&Characteristic degrees\\\hline
    $A_n$&$SU(n+1)$&$n$&$\Sigma_{n+1}$&$(n+1)!$&$2,3,\ldots,n+1$\\
    $B_n$&$SO(2n+1)$&$n$&$B_n$&$2^nn!$&$2,4,\ldots,2n$\\
    $C_n$&$Sp(n)$&$n$&$B_n$&$2^nn!$&$2,4,\ldots,2n$\\
    $D_n$&$SO(2n)$&$n$&$B_n^+$&$2^{n-1}n!$&$2,4,\ldots,2n-2,n$\\
    $G_2$&$G_2$&$2$&-----&$12$&$2,6$\\
    $F_4$&$F_4$&$4$&-----&$1152$&$2,6,8,12$\\
    $E_6$&$E_6$&$6$&-----&$51840$&$2,5,6,8,9,12$\\
    $E_7$&$E_7$&$7$&-----&$2903040$&$2,6,8,10,12,14,18$\\
    $E_8$&$E_8$&$8$&-----&$696729600$&$2,8,12,14,16,20,24,30$\\\hline
  \end{tabular}
\end{table}

Now we give a short alternative proof of the following theorem due to Ramras and Stafa \cite{RS1}, the formula \eqref{Ramras-Stafa intro} in Section \ref{Introduction}, using Theorem \ref{Baird}.

\begin{thm}\label{Ramras-Stafa}
  The Poincar\'e series of $\Hom(\Z^m,G)_1$ is given by
  $$P(\Hom(\Z^m,G)_1;t)=\frac{1}{|W|}\prod_{i=1}^r(1-t^{2d_i})\sum_{w\in W}\frac{\det(1+tw)^m}{\det(1-t^2w)}$$
  where $d_1,\ldots,d_n$ are the characteristic degrees of $W$.
\end{thm}

\begin{proof}
  Let us work over a field $\F$ of characteristic zero or prime to $|W|$. By the Shephard Todd theorem, there is an isomorphism of $W$-modules
  $$\P(n)\cong\P(n)^W\otimes\P(n)_W$$
  where $\P(n)_W$ denotes the ring of coinvariants of $W$. Then since $H^*(G/T;\F)$ is identified with $\P(n)_W$, there is an isomorphism
  $$(H^*(G/T;\F)\otimes H^*(T)^{\otimes m})^W\cong(\P(n)\otimes H^*(T)^{\otimes m})^W/(a_{d_1},\ldots,a_{d_n})$$
  where $a_1,\ldots,a_n$ are as in \eqref{polynomial invariant}. Since the sequence $a_1,\ldots,a_n$ is regular in $\P(n)$, it is also regular in $(\P(n)\otimes H^*(T)^{\otimes m})^W$. Then it follows from Theorem \ref{Baird} that
  $$P(\Hom(\Z^m,G)_1;t)=P((\P(n)\otimes H^*(T)^{\otimes m})^W;t)\prod_{i=1}^n(1-t^{2d_i}).$$
  By the standard argument in representation theory of finite groups,
  \begin{align*}
    P((\P(n)\otimes H^*(T)^{\otimes m})^W;t)&=\frac{1}{|W|}\sum_{w\in W}\sum_{i=0}^\infty\tr(w\vert_{(\P(n)\otimes H^*(T)^{\otimes m})^i})t^i\\
    &=\frac{1}{|W|}\sum_{w\in W}\left(\sum_{i=0}^\infty\tr(w\vert_{\P(n)^{2i}})t^{2i}\right)\left(\sum_{i=0}^\infty\tr(w\vert_{H^i(T;\F)})t^i\right)^m
  \end{align*}
  and it is easy to see that
  $$\sum_{i=0}^\infty\tr(w\vert_{\P(n)^{2i}})t^{2i}=\frac{1}{\det(1-t^2w)}\quad\text{and}\quad\sum_{i=0}^\infty\tr(w\vert_{H^i(T;\F)})t^i=\det(1+tw)$$
  where $A^i$ denote the $i$-dimensional part of a graded algebra $A$. Thus the proof is complete
\end{proof}


\section{Integer partions and permutations}\label{Integer partitions and permutations}

In this section, we assign (signed) integer partitions to (signed) permutations, and show enumerative properties of this assignment. The result in this section will be used to compute the Poincar\'e series of $\Hom(\Z^m,G)_1$ for the classical group $G$ in the next section.


\subsection{Integer partitions}

A \emph{partition} of a positive integer $n$ is a sequence of positive integers $\lambda=(\lambda_1,\ldots,\lambda_k)$ such that $\lambda_1\le\cdots\le\lambda_k$ and $\lambda_1+\cdots+\lambda_k=n$. For this integer partition $\lambda$, let
$$\ell(\lambda)\coloneqq k.$$
We denote an integer partition $\mu=(\underbrace{\mu_1,\ldots,\mu_1}_{n_1},\ldots,\underbrace{\mu_l,\ldots,\mu_l}_{n_l})$ by $(\mu_1^{n_1},\ldots,\mu_l^{n_l})$ alternatively. For this integer partition $\mu$, let
$$\theta(\mu)\coloneqq\mu_1\cdots\mu_ln_1!\cdots n_l!.$$
If $\lambda$ is the empty partition, then we set $\theta(\lambda)=1$. If $\lambda$ is a partition of $n$, we write
$$\lambda\vdash n.$$
We mean by $\lambda\vdash 0$ that $\lambda$ is the empty partition. For $\lambda\vdash n$ and $\mu\vdash k$ with $n\ge k$, let $\Emb(\mu,\lambda)$ denote the set of subsequences of $\lambda$ which are equal to $\mu$.

\begin{exam}
  If $\lambda=(1^2,2^2,3,4^2)$ and $\mu=(1,2,4)$, then $|\Emb(\mu,\lambda)|=2^3=8$.
\end{exam}

For a cyclic permutation $c\in\Sigma_n$, let $|c|$ denote the order of $c$. Let $w\in\Sigma_n$. Then there are disjoint cyclic permutations $c_1,\ldots,c_k\in\Sigma_n$ such that
$$w=c_1\cdots c_k\quad\text{and}\quad|c_1|+\cdots+|c_k|=n$$
which is called a cycle decomposition. Note that we do not omit 1-cycles in a cycle decomposition. A cycle decomposition $w=c_1\cdots c_k$ is called standard if $|c_1|\le\cdots\le|c_k|$ and the least element of $c_i$ is less than the least element of $c_{i+1}$ whenever $|c_i|=|c_{i+1}|$. Clearly, every permutation has a unique standard cycle decomposition.

\begin{exam}
  A cycle decomposition $(1\;6)(2\;5)(3\;4\;7)$ is standard, but $(2\;5)(1\;6)(3\;4\;7)$ is not standard.
\end{exam}

Let $w\in\Sigma_n$ with the standard cycle decomposition $w=c_1\cdots c_k$. Then we assign to $w$ a partition of $n$
$$c(w)\coloneqq(|c_1|,\ldots,|c_k|).$$
We show a property of $c(w)$ that we are going to use. Let $a\brack b$ denote the Stirling number of the first kind.

\begin{lem}\label{counting}
  For $i\ge 0$ and $\lambda\vdash k$ with $k\le n$,
  $$\sum_{\substack{w\in\Sigma_n\\\ell(c(w))-\ell(\lambda)=i}}|\Emb(\lambda,c(w))|=\frac{n!}{\theta(\lambda)(n-k)!}{n-k\brack i}.$$
\end{lem}

\begin{proof}
  Let $\lambda=(\lambda_1,\ldots,\lambda_k)$. Every subsequence of $c(w)$ which is equal to $\lambda$ is obtained by the following constructions:
  \begin{enumerate}
    \item we choose $k$ disjoint cyclic permutations of order $\lambda_1,\ldots,\lambda_k$, and
    \item we divide the remaining subset of $[n]$ by disjoint $i$ cycles.
  \end{enumerate}
  Since the number of cyclic permutations of order $d$ in $\Sigma_n$ is $\frac{n!}{d(n-d)!}$, the number of the first choices is $\frac{n!}{\theta(\lambda)(n-k)!}$. Since the Stirling number $a\brack b$ counts the number of permutations on $a$ letters having a cycle decomposition by $b$ cyclic permutations, the number of the second divisions is $n-k\brack i$. Thus the identity in the statement holds.
\end{proof}

We record an identity involving Stirling numbers that we will use later. See \cite[Proposition 1.3.4]{S} for the proof.

\begin{lem}\label{Stirling}
  There is an equality
  $$\sum_{k=1}^n{n\brack k}x^k=x(x+1)\cdots(x+n-1).$$
\end{lem}


\subsection{Signed integer partitions}

A \emph{signed partition} of a positive integer $n$ is an ordered sequence of non-zero integers $\lambda=(\lambda_1,\ldots,\lambda_k)$ such that $(|\lambda_1|,\ldots,|\lambda_k|)$ is a partition of $n$.

\begin{exam}
  Sequences $(1,-1,2)$ and $(-1,1,2)$ are distinct signed partitions of $4$, because signed partitions are ordered sequences.
\end{exam}

If $\lambda$ is a signed partition of $n$, then we write
$$\lambda\vdash\pm n.$$
For $\lambda=(\lambda_1,\ldots,\lambda_k)\vdash\pm n$, we define
$$\lambda^+=(|\lambda_1|,\ldots,|\lambda_k|)\quad\text{and}\quad\sgn(\lambda)=\sgn(\lambda_1)\cdots\sgn(\lambda_k).$$
For $\lambda\vdash\pm n$ and $\mu\vdash\pm k$ with $n\ge k$, let $\Emb(\mu,\lambda)$ denote the set of ordered subsequences of $\lambda$ which are equal to $\mu$.

\begin{exam}
  If $\lambda=(1,-1,1,2,2,3,-3,4)$ and $\mu=(-1,1,2,3)$, then $|\Emb(\mu,\lambda)|=2$.
\end{exam}

If $|\Emb(\mu,\lambda)|>0$, then by choosing a specific ordered subsequence $\bar{\mu}$ of $\lambda$ which is equal to $\mu$, we define $\sgn(\mu,\lambda)$ to be the product of signs of elements of $\lambda$ which are not in $\mu$. Note that $\sgn(\mu,\lambda)$ does not depend on the choice of $\bar{\mu}$.

\begin{exam}
  If $\lambda=(1,-1,1,-1,1,2,-2,3,-3,4)$ and $\mu=(-1,1,2,3)$, then $\sgn(\mu,\lambda)=(-1)^3=-1$.
\end{exam}

A signed permutation on $n$ letters is a permutation $w$ on $\{\pm 1,\pm 2,\ldots,\pm n\}$ such that $w(-i)=-w(i)$. Then signed permutations on $n$ letters form a group, which we denote by $B_n$. Clearly, $B_n$ is isomorphic with a semidirect product $(\Z/2)^n\rtimes\Sigma_n$. Then to each $w\in B_n$ one can assign unique $s\in(\Z/2)^n$ and $v\in\Sigma_n$ such that $w=sv$. In this case, let
$$w^+\coloneqq v\quad\text{and}\quad\sgn(w)\coloneqq s_1\cdots s_n$$
where $s=(s_1,\ldots,s_n)\in(\Z/2)^n=\{-1,+1\}^n$. Let $B_n^+$ denote the subgroup of $B_n$ consisting of $w\in B_n$ with $\sgn(w)=1$.

It is obvious that every signed permutation $w\in B_n$ admits a decomposition $w=c_1\cdots c_k$ by disjoint signed cyclic permutations. A signed cycle decomposition $w=c_1,\cdots c_k$ is said to be standard if $w^+=c_1^+\cdots c_k^+$ is the standard cycle decomposition. Then every signed permutation admits a unique standard signed cycle decomposition.

To every signed permutation $w\in B_n$ with the standard signed cycle decomposition $w=c_1\cdots c_k$, we assign a signed partition of $n$
$$c(w)\coloneqq(\sgn(c_1)|c_1^+|,\ldots,\sgn(c_k)|c_k^+|).$$

\begin{lem}\label{counting signed}
  \begin{enumerate}
    \item For $i\ge 0$ and $\lambda\vdash\pm k$ with $k\le n$,
    $$\sum_{\substack{w\in B_n\\\ell(c(w)^+)-\ell(\lambda^+)=i}}|\Emb(\lambda,c(w))|=\frac{2^{n-\ell(\lambda^+)}n!}{\theta(\lambda^+)(n-k)!}{n-k\brack i}.$$
    \item Let $\lambda,\lambda'$ be signed partitions of $n-1$ such that all but one elements are the same. Then
    $$\sum_{w\in B_n^+}|\Emb(\lambda,c(w))|=\sum_{w\in B_n^+}|\Emb(\lambda',c(w))|.$$
  \end{enumerate}
\end{lem}

\begin{proof}
  The first identity is proved by the same argument as in the proof of Lemma \ref{counting} with an easy modification about signs. Let $\lambda,\lambda'$ be as in the second statement. Then each element of $\Emb(\lambda,c(w))$ and $\Emb(\lambda',c(w))$ are obtained by adding $\pm 1$ to $\lambda$ and $\lambda'$, respectively. Since all but one elements of $\lambda,\lambda'$ are the same, the difference is by sign. Thus the second identity is obtained.
\end{proof}


\section{Poincar\'e series}\label{Poincare series}

In this section, we refine the foumula in Theorem \ref{Ramras-Stafa} for the classical groups by applying the results in the previous section, where the refined formula is given in terms of (signed) integer partitions.


\subsection{Unitary groups}

We start with unitary groups. For a positive integer $k$, let
$$q_k^m(t)=(-1)^{m(k-1)}t^{(m-2)k}+\frac{(1+(-1)^{k+1}t^k)^m}{1-t^{2k}}.$$
For $\lambda=(\lambda_1,\ldots,\lambda_l)\vdash k\le n$, let
$$q^{m,n}_\lambda(t)=t^{(m-2)(n-k)}q_{\lambda_1}^m(t)\cdots q_{\lambda_l}^m(t)$$
where we set $q_\lambda^{m,n}(t)=t^{(m-2)n}$ if $\lambda$ is the empty partition.

\begin{thm}[Theorem \ref{Poincare intro}]\label{Poincare U}
  The Poincar\'e series of $\Hom(\Z^m,U(n))_1$ is given by
  $$P(\Hom(\Z^m,U(n))_1;t)=
  \begin{cases}
    \displaystyle\prod_{i=1}^n(1-t^{2i})\sum_{k=n-1}^n\sum_{\lambda\vdash k}\frac{(-1)^{n+k}}{\theta(\lambda)}q_\lambda^{m,n}(t)&(m\text{ even})\\
    \displaystyle\prod_{i=1}^n(1-t^{2i})\sum_{k=0}^n\sum_{\lambda\vdash k}\frac{(-1)^k}{\theta(\lambda)}q_\lambda^{m,n}(t)&(m\text{ odd}).
  \end{cases}$$
\end{thm}

\begin{proof}
  The Weyl group of $U(n)$ is $\Sigma_n$ acting canonically on $\F^n$. Then we apply the formula of Theorem \ref{Ramras-Stafa} to the canonical representation of $\Sigma_n$. Let $w\in\Sigma_n$ with the standard cycle decomposition $w=c_1\cdots c_k$. Since $\det(1+tc)=1+(-1)^{|c|+1}t^{|c|}$ for a cyclic permutation $c\in\Sigma_n$ and $c_1,\ldots,c_k$ are disjoint,
  $$\frac{\det(1+tw)^m}{\det(1-t^2w)}=\prod_{i=1}^k\frac{\left(1+(-1)^{|c_i|+1}t^{|c_i|}\right)^m}{1-t^{2|c_i|}}=\prod_{i=1}^k((-1)^{m(|c_i|-1)+1}t^{(m-2)|c_i|}+q_{|c_i|}^m(t)).$$
  Then it follows from Lemma \ref{counting} that
  \begin{align*}
    \sum_{w\in\Sigma_n}\frac{\det(1+tw)^m}{\det(1-t^2w)}&=\sum_{k=0}^n\sum_{\lambda\vdash k}\sum_{i=0}^{n-k}\sum_{\substack{w\in\Sigma_n\\\ell(c(w))-\ell(\lambda)=i}}(-1)^{mk+(m+1)i}|\Emb(\lambda,c(w))|q_\lambda^{m,n}(t)\\
    &=\sum_{k=0}^n\sum_{\lambda\vdash k}\sum_{i=0}^{n-k}(-1)^{mk+(m+1)i}\frac{n!}{\theta(\lambda)(n-k)!}{n-k\brack i}q_\lambda^{m,n}(t)
  \end{align*}
  By Lemma \ref{Stirling}, $\sum_{i=1}^{n-k}(-1)^i{n-k\brack i}=0$ for $n-k\ge 2$ and $\sum_{i=1}^{n-k}{n-k\brack i}=(n-k)!$. Thus by Theorem \ref{Ramras-Stafa} the proof is complete.
\end{proof}

Applying Theorem \ref{Poincare U}, we give a formula for the Poincar\'e series of $\Hom(\Z^m,SU(n))_1$.

\begin{lem}\label{covering}
  Let $G,H$ be compact connected Lie groups, and let $\F$ be a field of characteristic zero or prime to the Weyl group of $G$. If there is a covering $G\to H$, then
  $$H^*(\Hom(\Z^m,G)_1;\F)\cong H^*(\Hom(\Z^m,H)_1;\F).$$
\end{lem}

\begin{proof}
  Since there is a covering $G\to H$, the Weyl groups of $G$ and $H$ have the same order. Then the characteristic of $\F$ is also prime to the order of the Weyl group of $H$, and so the Weyl groups of $G$ and $H$ are isomorphic as reflection groups over $\F$. Then by Theorem \ref{Baird} below, the proof is done.
\end{proof}

\begin{cor}\label{Poincare SU}
  The Poincar\'e series of $\Hom(\Z^m,SU(n))_1$ is given by
  $$P(\Hom(\Z^m,SU(n))_1;t)=
  \begin{cases}
    \displaystyle\prod_{i=1}^n(1-t^{2i})\sum_{k=n-1}^n\sum_{\lambda\vdash k}\frac{(-1)^{n+k}}{\theta(\lambda)(1+t)^m}q_\lambda^{m,n}(t)&(m\text{ even})\\
    \displaystyle\prod_{i=1}^n(1-t^{2i})\sum_{k=0}^n\sum_{\lambda\vdash k}\frac{(-1)^k}{\theta(\lambda)(1+t)^m}q_\lambda^{m,n}(t)&(m\text{ odd}).
  \end{cases}$$
\end{cor}

\begin{proof}
  Since $U(n)=(SU(n)\times S^1)/(\Z/n)$, it follows from Lemma \ref{covering} that
  \begin{equation}
    \label{U-SU}
    H^*(\Hom(\Z^m,U(n))_1;\F)\cong H^*(\Hom(\Z^m,SU(n))_1;\F)\otimes H^*((S^1)^m;\F)
  \end{equation}
  where $\Hom(\Z^m,S^1)_1=(S^1)^m$. Then by Theorem \ref{Poincare U} the proof is complete.
\end{proof}

\begin{exam}\label{example Poincare series}
  We give example calculations of the Poincar\'e series of $\Hom(\Z^2,SU(n))_1$ by using Theorem \ref{Poincare SU}.

  \begin{align*}
    P(\Hom(\Z^2,SU(2))_1;t)&=1+{t^{2}}+2 {t^{3}}\\
    P(\Hom(\Z^2,SU(3))_1;t)&=1+{t^{2}}+2 {t^{3}}+2 {t^{4}}+4 {t^{5}}+{t^{6}}+2 {t^{7}}+3 {t^{8}}\\
    P(\Hom(\Z^2,SU(4))_1;t)&=1+{t^{2}}+2 {t^{3}}+2 {t^{4}}+4 {t^{5}}+4 {t^{6}}+8 {t^{7}}+6 {t^{8}}+6 {t^{9}}+8 {t^{10}}\\
    &\quad+6 {t^{11}}+7 {t^{12}}+2 {t^{13}}+3 {t^{14}}+4 {t^{15}}\\
    P(\Hom(\Z^2,SU(5))_1;t)&=1+{t^{2}}+2 {t^{3}}+2 {t^{4}}+4 {{x}^{5}}+4 {t^{6}}+8 {t^{7}}+10 {t^{8}}+14 {t^{9}}\\
    &\quad+13 {t^{10}}+16 {t^{11}}+22 {t^{12}}+18 {t^{13}}+21 {t^{14}}+20 {t^{15}}+22 {t^{16}}\\
    &\quad+18 {t^{17}}+14 {t^{18}}+14 {t^{19}}+10 {t^{20}}+10 {t^{21}}+3 {t^{22}}+4 {t^{23}}+5 {t^{24}}
  \end{align*}

  \begin{align*}
    P(\Hom(\Z^2,SU(6))_1;t)&=1+{t^{2}}+2 {t^{3}}+2 {t^{4}}+4 {t^{5}}+4 {t^{6}}+8 {t^{7}}+10 {t^{8}}+14 {t^{9}}+18 {t^{10}}\\
    &\quad+26 {t^{11}}+29 {t^{12}}+32 {t^{13}}+43 {t^{14}}+46 {t^{15}}+54 {t^{16}}+54 {t^{17}}\\
    &\quad+62 {t^{18}}+66 {t^{19}}+65 {t^{20}}+62 {t^{21}}+59 {t^{22}}+64 {t^{23}}+52 {t^{24}}\\
    &\quad+48 {t^{25}}+44 {t^{26}}+40 {t^{27}}+30 {t^{28}}+22 {t^{29}}+20 {t^{30}}+14 {t^{31}}\\
    &\quad+13 {t^{32}}+4 {t^{33}}+5 {t^{34}}+6 {t^{35}}
  \end{align*}

  Observe the following two things from these Poincar\'e series.
  \begin{enumerate}
    \item The coefficient of the top term of the Poincar\'e series of $\Hom(\Z^2,SU(n))_1$ is $n$ for $2\le n\le 6$.

    \item The Poincar\'e series of $\Hom(\Z^2,SU(n))_1$ and $\Hom(\Z^2,SU(n+1))_1$ are equal in dimension $2n-1$ for $2\le n\le 5$.
  \end{enumerate}

  The first observation will be proved to be true for every $n$ in Section \ref{Top terms}, and the second will be justified by homological stability proved in Section \ref{Homological stability}.
\end{exam}

\begin{rem}
  The homology dimension of $\Hom(\Z^2,G)_1$ can be greater than the degree of its Poincar\'e series. For example, as is observed in \cite{C}, there is a stable homotopy decomposition
  $$\Hom(\Z^2,SU(2))_1\simeq_sS^2\vee S^3\vee S^3\vee\Sigma^2\R P^2,$$
  implying that the homology dimension of $\Hom(\Z^2,SU(2))_1$ is four. But the degree of its Poincar\'e series is three. This occurs because we are taking the ground field of characteristic zero or greater than two for the Poincar\'e series.
\end{rem}


\subsection{Symplectic and special orthogonal groups}

We refine the formula for symplectic and special orthogonal groups. For a nonzero integer $k$, let
$$q_k^m(t)=(-1)^{m(k+1)}\sgn(k)^{m+1}t^{(m-2)|k|}+\frac{(1+(-1)^{k+1}\sgn(k)t^{|k|})^m}{1-\sgn(k)t^{2|k|}}$$
and for $\lambda=(\lambda_1,\ldots,\lambda_l)\vdash\pm k$ with $k\le n$, let
$$q_\lambda^{m,n}=t^{(m-2)(n-|k|)}q_{\lambda_1}^m(t)\cdots q_{\lambda_l}^m(t).$$
We set $q_\lambda^{m,n}=t^{(m-2)n}$ for the empty partition $\lambda$ as in the case of $U(n)$.

\begin{thm}\label{Poincare Sp}
  The Poincar\'e series of $\Hom(\Z^m,Sp(n))_1$ is given by
  $$P(\Hom(\Z^m,Sp(n))_1;t)=
  \begin{cases}
    \displaystyle\prod_{i=1}^n(1-t^{4i})\sum_{\lambda\vdash\pm n}\frac{1}{2^{\ell(\lambda^+)}\theta(\lambda^+)}q_\lambda^{m,n}(t)&(m\text{ even})\\
    \displaystyle\prod_{i=1}^n(1-t^{4i})\sum_{k=0}^n\sum_{\lambda\vdash\pm k}\frac{(-1)^{n-k}}{2^{\ell(\lambda^+)}\theta(\lambda^+)}q_\lambda^{m,n}(t)&(m\text{ odd}).
  \end{cases}$$
\end{thm}

\begin{proof}
  The Weyl group of $Sp(n)$ is $B_n$ acting canonically on $\F^n$. Then we apply the formula of Theorem \ref{Ramras-Stafa} to the canonical representation of $B_n$. Then since $\det(1+tc)=1+(-1)^{|c|+1}\sgn(c)t^{|c|}$ for a signed cyclic permutation $c\in B_n$,
  for $w\in B_n$ with the standard cycle decomposition $w=c_1\cdots c_k$, we have
  \begin{align*}
    &\frac{\det(1+tw)^m}{\det(1-t^2w)}\\
    &=\prod_{i=1}^k((-1)^{m(|c_i^+|+1)+1}\sgn(c_i)^{m+1}t^{(m-2)|c_i^+|}+q^m_{\sgn(c_i)|c_i|}(t))\\
    &=\sum_{k=0}^n\sum_{i=0}^{n-k}\sum_{\substack{\lambda\vdash\pm k\\\ell(c(w)^+)-\ell(\lambda^+)=i}}(-1)^{m(n-k)+(m+1)i}\sgn(\lambda,c(w))^{m+1}|\Emb(\lambda,c(w))|q_\lambda^{m,n}(t).
  \end{align*}
  It is obvious that
  $$\sum_{s\in (\Z/2)^n}\sum_{\substack{\lambda\vdash\pm k\\\ell(c(sw)^+)-\ell(\lambda^+)=i}}\sgn(\lambda,c(sw))|\Emb(\lambda,c(sw))|=0$$
  unless $i=0$.
  Then it follows from Lemma \ref{counting signed} that for $m$ even
  $$\sum_{w\in B_n}\frac{\det(1+tw)^m}{\det(1-t^2w)}=\sum_{\lambda\vdash\pm n}\frac{2^{n-\ell(\lambda^+)}}{\theta(\lambda^+)}q_\lambda^{m,n}(t).$$
  For $m$ odd, by Lemmas \ref{Stirling} and \ref{counting signed},
  \begin{align*}
    \sum_{w\in B_n}\frac{\det(1+tw)^m}{\det(1-t^2w)}&=\sum_{k=0}^n\sum_{i=0}^{n-k}\sum_{\lambda\vdash\pm k}\frac{(-1)^{n-k}2^{n-\ell(\lambda^+)}n!}{\theta(\lambda^+)(n-k)!}{n-k\brack i}q_\lambda^{m,n}(t)\\
    &=\sum_{k=0}^n\sum_{\lambda\vdash\pm k}\frac{(-1)^{n-k}2^{n-\ell(\lambda^+)}n!}{\theta(\lambda^+)}q_\lambda^{m,n}(t).
  \end{align*}
  Thus by Theorem \ref{Ramras-Stafa} the proof is complete.
\end{proof}

\begin{thm}\label{cohomology SO Sp}
  Let $\F$ be a field of characteristic zero or prime to $2(n!)$. Then
  \[
    H^*(\Hom(\Z^m,Sp(n))_1;\F)\cong H^*(\Hom(\Z^m,SO(2n+1))_1;\F).
  \]
\end{thm}

\begin{proof}
  The Weyl groups of $Sp(n)$ and $SO(2n+1)$ are isomorphic as reflection groups over a field of characteristic $>2$. Then by Theorem \ref{Baird}, we obtain the isomorphism in the statement.
\end{proof}

The following corollary is immediate from Theorem \ref{cohomology SO Sp}

\begin{cor}\label{Poincare SO(2n+1)}
  $P(\Hom(\Z^m,SO(2n+1))_1;t)=P(\Hom(\Z^m,Sp(n))_1;t)$.
\end{cor}

Then by Theorem \ref{Poincare Sp}, we obtain a formula for the Poincar\'e series of $\Hom(\Z^m,SO(2n+1))_1$.

\begin{thm}\label{Poincare SO(2n)}
  The Poincar\'e series of $\Hom(\Z^m,SO(2n))_1$ is given by
  $$(1-t^{2n})\prod_{i=1}^{n-1}(1-t^{4i})\left(\sum_{\lambda\vdash\pm(n-1)}\frac{-\sgn(\lambda)}{2^{\ell(\lambda^+)}\theta(\lambda^+)}q_\lambda^{m,n}+\sum_{\lambda\vdash\pm n}\frac{1}{2^{\ell(\lambda^+)}\theta(\lambda^+)}q_\lambda^{m,n}\right)$$
  for $m$ even and
  $$(1-t^{2n})\prod_{i=1}^{n-1}(1-t^{4i})\sum_{k=0}^n\sum_{\lambda\vdash\pm k}\frac{(-1)^{n-k}}{2^{\ell(\lambda^+)}\theta(\lambda^+)}q_\lambda^{m,n}$$
  for $m$ odd.
\end{thm}

\begin{proof}
  The Weyl group of $SO(2n)$ is $B_n^+$ with the canonical representation. Then we apply the formula of Theorem \ref{Ramras-Stafa} to the canonical representation of $B_n^+$. As in the case of $Sp(n)$, we can show the identity
  \begin{align*}
    &\frac{\det(1+tw)^m}{\det(1-t^2w)}\\
    &=\sum_{k=0}^n\sum_{i=0}^{n-k}\sum_{\substack{\lambda\vdash\pm k\\\ell(c(w)^+)-\ell(\lambda^+)=i}}(-1)^{m(n-k)+(m+1)i}\sgn(\lambda,c(w))^{m+1}|\Emb(\lambda,c(w))|q_\lambda^{m,n}(t)
  \end{align*}
  for $w\in B_n^+$. By Lemma \ref{counting signed},
  $$\sum_{\substack{w\in B_n^+\\\ell(c(w)^+)-\ell(\lambda^+)=i}}|\Emb(\lambda,c(w))|=\frac{2^{n-\ell(\lambda^+)-1}n!}{\theta(\lambda^+)(n-k)!}{n-k\brack i}.$$
  Since $\sgn(c(w))=\sgn(w)=1$ for $w\in B_n^+$, $\sgn(\lambda,c(w))=\sgn(\lambda)$.
  Then it follows that
  \begin{align*}
    &\sum_{w\in B_n^+}\frac{\det(1+tw)^m}{\det(1-t^2w)}\\
    &=\sum_{k=0}^n\sum_{i=0}^{n-k}\sum_{\substack{\lambda\vdash\pm k\\\ell(c(w)^+)-\ell(\lambda^+)=i}}(-1)^{m(n-k)+(m+1)i}\sgn(\lambda)^{m+1}\frac{2^{n-\ell(\lambda^+)-1}n!}{\theta(\lambda^+)(n-k)!}{n-k\brack i}q_\lambda^{m,n}.
  \end{align*}
  By Lemma \ref{Stirling}, $\sum_{i=1}^{n-k}(-1)^i{n-k\brack i}=0$ for $n-k\ge 2$. Hence for $m$ even,
  $$\sum_{w\in B_n^+}\frac{\det(1+tw)^m}{\det(1-t^2w)}=-\sum_{\lambda\vdash\pm n-1}\frac{\sgn(\lambda)2^{n-\ell(\lambda^+)-1}n!}{\theta(\lambda^+)}q_\lambda^{m,n}+\sum_{\lambda\vdash\pm n}\frac{2^{n-\ell(\lambda^+)-1}n!}{\theta(\lambda^+)}q_\lambda^{m,n}.$$
  For $m$ odd, by Lemmas \ref{Stirling} and \ref{counting signed},
  $$\sum_{w\in B_n^+}\frac{\det(1+tw)^m}{\det(1-t^2w)}=\sum_{k=0}^n\sum_{\lambda\vdash\pm k}\frac{(-1)^{n-k}2^{n-\ell(\lambda^+)-1}n!}{\theta(\lambda^+)}q_\lambda^{m,n}.$$
  Thus by Theorem \ref{Ramras-Stafa} the proof is complete.
\end{proof}


\section{Top terms}\label{Top terms}

In this section, we determine the top term of the Poincar\'e series of $\Hom(\Z^m,G)_1$ for the classical group $G$ by applying the formulae obtained in the previous section. We also determine the top term of the Poincar\'e series of $\Hom(\Z^2,G)_1$ for any compact connected Lie group $G$ by combining the results for the classical group and a computer calculation for the exceptional groups. This has an application to the rational homotopy of $\Hom(\Z^m,G)_1$.


\subsection{Classical groups}

We determine the top term of the Poincar\'e series of $\Hom(\Z^m,G)_1$ for the classical group $G$. Define
$$c(m,G)=
\begin{cases}
  m+n-2&(G=U(n),SU(n))\\
  m+n-1&(G=Sp(n),SO(2n+1),SO(2n))
\end{cases}$$
and
$$d(m,G)=
\begin{cases}
  n^2+(m-1)n&(G=U(n))\\
  n^2+(m-1)(n-1)-1&(G=SU(n))\\
  2n^2+mn&(G=Sp(n),SO(2n+1))\\
  2n^2+(m-2)n&(G=SO(2n)).
\end{cases}$$
We also set $\epsilon(U(n))=1$ and $\epsilon(G)=0$ for other classical groups $G$.

\begin{thm}[Theorem \ref{top term intro}]\label{top term classical}
  Let $G$ be the classical group, where we assume $n>2$ for $G=SO(2n)$. The top term of the Poincar\'e series of $\Hom(\Z^m,G)_1$ is
  $$\begin{cases}
    t^{d(m,G)}&(m\text{ odd})\\
    \binom{c(m,G)}{m-1}t^{d(m-1,G)+\epsilon(G)}&(m\text{ even}.)
  \end{cases}$$
\end{thm}

\begin{proof}
  We define the degree of a rational function $\frac{P(t)}{Q(t)}$ by $\deg P(t)-\deg Q(t)$, where $P(t),Q(t)$ are polynomials.

  \noindent\textbf{Unitary groups.} Suppose $m$ is odd. Since $q_\lambda^{m,n}$ is of degree $(m-2)n-k$ for $\lambda\vdash k$, it follows from Theorem \ref{Poincare U} that the Poincar\'e series of $\Hom(\Z^m,U(n))_1$ is of degree at most $n^2+(m-1)n$ such that the degree $n^2+(m-1)n$ term is included in $\frac{1}{\theta(\lambda)}q_\lambda^{m,n}(t)$ for the empty partition $\lambda$. Then its coefficient is 1, and so the top term of the Poincar\'e series is $t^{n^2+(m-1)n}$.

  Suppose $m$ is even. Then by Theorem \ref{Poincare U} the Poincar\'e series is of degree at most $n^2+(m-2)n+1$. Going back to the calculation of the Poincar\'e series, one can see that the coefficient of the degree $n^2+(m-2)n+1$ term is
  $$\frac{1}{n!}\sum_{\lambda\vdash n-1}m^{\ell(\lambda)}\sum_{w\in\Sigma_n}|\Emb(\lambda,c(w))|=\frac{1}{n!}\sum_{i=1}^{n-1}m^i\sum_{\substack{\lambda\vdash n-1\\\ell(\lambda)=i}}\sum_{w\in\Sigma_n}|\Emb(\lambda,c(w))|.$$
  By Lemma \ref{counting} for $k=1$,
  $$\sum_{\substack{\lambda\vdash n-1\\\ell(\lambda)=i}}\sum_{w\in\Sigma_n}|\Emb(\lambda,c(w))|=n{n-1\brack i}$$
  and by Lemma \ref{Stirling}, $\sum_{i=1}^{n-1}m^i{n-1\brack i}=\frac{(m+n-2)!}{(m-1)!}$. Thus the coefficient of the degree $n^2+(m-2)n+1$ term is $\binom{m+n-2}{m-1}$.

  \noindent\textbf{Special unitary groups.} This follows from \eqref{U-SU} and the case of $U(n)$ above.

  \noindent\textbf{Symplectic groups.} For $m$ odd, it is immediate from Theorem \ref{Poincare Sp} that the top term of the Poincar\'e series of $\Hom(\Z^m,Sp(n))_1$ is $t^{2n^2+mn}$. Suppose that $m$ is even. Then by Theorem \ref{Poincare Sp}, the Poincar\'e series is of degree $\le 2n^2+(m-1)n$. Going back to the calculation of the Poincar\'e series, we can see that the coefficient of the degree $2n^2+(m-1)n$ term is
  $$\frac{1}{2^nn!}\sum_{\lambda\vdash\pm n}m^{\ell(\lambda^+)}\sum_{w\in B_n}|\Emb(\lambda,c(w))|=\frac{1}{2^nn!}\sum_{i=1}^nm^i\sum_{\substack{\lambda\vdash\pm n\\\ell(\lambda^+)=i}}\sum_{w\in B_n}|\Emb(\lambda,c(w))|.$$\
  By Lemma \ref{counting signed} (1) for $k=0$,
  $$\sum_{\substack{\lambda\vdash\pm n\\\ell(\lambda^+)=i}}\sum_{w\in B_n}|\Emb(\lambda,c(w))|=2^n{n\brack i}.$$
  Thus by Lemma \ref{Stirling}, the coefficient is $\binom{m+n-1}{m-1}$.

  \noindent\textbf{Special orthogonal groups.} By Theorem \ref{cohomology SO Sp}, the $SO(2n+1)$ case follows from the $Sp(n)$ case above. Then we consider the $SO(2n)$ case. For $m$ odd, it is immediate to see that the top term is $t^{2n^2+(m-2)n}$. For a nonzero integer $k$, $q_k(t)-q_{-k}(t)$ is of degree $(m-4)k$. Then $\sum_{\lambda\vdash\pm(n-1)}\sgn(\lambda)q_\lambda^{m,n}$ is of degree at most $(m-4)n+2$. On the other hand, since $q_\lambda^{m,n}(t)$ is of degree $(m-3)k$ for $\lambda\vdash\pm k$, $\sum_{\lambda\vdash\pm n}q_\lambda^{m,n}$ is of degree at most $(m-3)n$. Then the top term is included in
  $$(1-t^{2n})\prod_{i=1}^{n-1}(1-t^{4i})\sum_{\lambda\vdash\pm n}\frac{1}{2^{\ell(\lambda^+)}\theta(\lambda^+)}q_\lambda^{m,n}$$
  because we are assuming $n>2$. Thus by the same calculation as in the $Sp(n)$ case above, one gets that the top term is $\binom{m+n-1}{m-1}t^{2n^2+(m-3)n}$.
\end{proof}

\begin{rem}\label{SO(4)}
  The top term of the Poincar\'e series of $\Hom(\Z^m,SO(2n))_1$ for $n=1,2$ is easily obtained as follows. Since $\Hom(\Z^m,SO(2))_1=(S^1)^m$, its Poincar\,e series is $(1+t)^m$. Then its top term is $t^m$. Since $SO(4)=(SU(2)\times SU(2))/(\Z/2)$, $P(\Hom(\Z^m,SO(4))_1;t)=P(\Hom(\Z^m,SU(2))_1;t)^2$ by Lemma \ref{covering}. Then the top term of the Poincar\'e series of $\Hom(\Z^m,SO(4))_1$ is $m^2t^{2m+2}$.
\end{rem}

Observe that Theorem \ref{top term classical} implies that the topology of $\Hom(\Z^m,G)_1$ depends on the parity of $m$. We give an alternative expression for this dependence. Antol\'in, Gritschacher and Villarreal (see \cite[Remark 6.5]{RS1}) observed that the Poincar\'e series of $\Hom(\Z^m,G)_1$ is palindromic for $m$ odd. On the other hand, it follows from Theorems \ref{top term classical} and Remark \ref{SO(4)} that when $G$ is the classical group which is neither trivial nor $S^1$, the coefficient of the top term of the Poincar\'e series of $\Hom(\Z^m,G)_1$ is greater than one for $m$ even. Thus we obtain the following.

\begin{cor}[Corollary \ref{palindromic intro}]\label{palindromic}
  Let $G$ be the classical group which is neither trivial nor $S^1$. Then the Poincar\'e series of $\Hom(\Z^m,G)_1$ is palindromic if and only if $m$ is odd.
\end{cor}


\subsection{General Lie groups}

We further determine the top term of $\Hom(\Z^2,G)_1$ for any compact connected Lie group $G$. We first observe the case of simple Lie groups, and then consider the general case.

\begin{cor}\label{top term simple}
  Let $G$ be a compact connected simple Lie group. Then the top term of the Poincar\'e series of $\Hom(\Z^2,G)_1$ is $(\mathrm{rank}\,G+1)t^{\dim G}$.
\end{cor}

\begin{proof}
  Combine Theorem \ref{top term classical} and Appendix.
\end{proof}

\begin{thm}[Theorem \ref{top term general intro}]\label{top term general}
  Let $G$ be a compact connected Lie group with simple factors $G_1,\ldots,G_k$. Then the top term of the Poincar\'e series of $\Hom(\Z^2,G)_1$ is
  $$(\mathrm{rank}\,G_1+1)\cdots(\mathrm{rank}\,G_k+1)t^{\dim G+\mathrm{rank}\,\pi_1(G)}.$$
\end{thm}

\begin{proof}
  There are compact simply-connected simple Lie groups $G_1,\ldots,G_k$ (simple factors in the statement), a torus $T$, and a discrete subgroup $K$ of the center of $G_1\times\cdots\times G_k\times T$ such that
  $$G\cong(G_1\times\cdots\times G_k\times T)/K.$$
  Then by Lemma \ref{covering},
  $$P(\Hom(\Z^2,G)_1;t)=\prod_{i=1}^kP(\Hom(\Z^2,G_i)_1;t).$$
  Since $\Hom(\Z^2,T)_1=T^2$, the top term of $P(\Hom(\Z^2,T)_1;t)$ is $t^{2\dim T}$. Then by Corollary \ref{top term simple}, the top term of $P(\Hom(\Z^2,G)_1;t)$ is
  $$(\mathrm{rank}\,G_1+1)\cdots(\mathrm{rank}\,G_k+1)t^{\dim G_1+\cdots+\dim G_k+2\dim T}.$$
  Since $\dim G=\dim G_1+\cdots+\dim G_k+\dim T$ and $\mathrm{rank}\,\pi_1(G)=\dim T$, the proof is complete.
\end{proof}

We consider an application of Theorem \ref{top term general} to rational homotopy. Let $X$ be a simply connected finite complex. Then by dichotomy in rational homotopy theory \cite{FHT}, $\sum_{n\ge 0}\pi_n(X)\otimes\Q$ is either finite or of exponential growth. In the former case, $X$ is called rationally elliptic, and the latter, rationally hyperbolic.

It is shown by Gome\'z, Pettet, and Souto \cite{GPS} that there is an isomorphism
$$\pi_1(\Hom(\Z^m,G)_1)\cong\pi_1(G)^m.$$
In particular, if $G$ is simply-connected, then so is $\Hom(\Z^m,G)_1$. On the other hand, since $\Hom(\Z^m,G)_1$ is a real algebraic variety, it is a finite complex. Then $\Hom(\Z^m,G)_1$ is either rationally elliptic or hyperbolic when $G$ is simply-connected. For $m=1$, $\Hom(\Z,G)_1=G$ which is rationally elliptic. For $m\ge 2$, one has:

\begin{cor}[Corollary \ref{hyperbolic intro}]\label{hyperbolic}
  If $G$ is a non-trivial compact simply-connected Lie group and $m\ge 2$, then $\Hom(\Z^m,G)_1$ is rationally hyperbolic.
\end{cor}

\begin{proof}
  Since $G$ is non-trivial and $m\ge 2$, the top term of the Poincar\'e series of $\Hom(\Z^2,G)_1$ has coefficient greater than one by Theorem \ref{top term general}. On the other hand, the rational cohomology of a rationally elliptic space satisfies the Poincar\'e duality. In particular, the top term of the Poincar\'e series has coefficient one. Thus $\Hom(\Z^2,G)_1$ must be rationally hyperbolic. By definition, if $X$ is a retract of $Y$ and is rationally hyperbolic, then $Y$ is rationally hyperbolic too. Since $\Hom(\Z^2,G)_1$ is a retract of $\Hom(\Z^m,G)_1$ for $m\ge 2$, the proof is complete.
\end{proof}


\section{Cohomology generators}\label{Cohomology generators}

In this section we give a minimal generating set of the cohomology of $\Hom(\Z^m,G)_1$ for the classical group $G$, except for $SO(2n)$. We set notation that are used throughout this section. Let $G$ be a compact connected Lie group of rank $n$, and let $W$ denote its Weyl group. We set
$$\widetilde{\H}(m,G)\coloneqq\P(n)\otimes\bigotimes_{i=1}^m\Lambda(y_1^i,\ldots,y_n^i)\quad\text{and}\quad\H(m,G)\coloneqq\P(n)_W\otimes\bigotimes_{i=1}^m\Lambda(y_1^i,\ldots,y_n^i)$$
where $\P(n)$ is as in Section \ref{Cohomology and Weyl groups}. By Theorem \ref{Baird},
$$H^*(\Hom(\Z^m,G)_1;\F)\cong\H(m,G)^W$$
where the ground field $\F$ is of characteristic zero or prime to $|W|$. Then we understand the results on $\H(m,G)^W$ below as those on $H^*(\Hom(\Z^m,G)_1;\F)$.


\subsection{Generators for unitary groups}

We assume that the ground field $\F$ is of characteristic prime to $n!$. We give generators of $\H(m,U(n))^{\Sigma_n}$. Let $e_i$ denote the $i$-th symmetric polynomial in $\P(n)$ for $1\le i\le n$. Then
$$\P(n)^{\Sigma_n}=\F[e_1,\ldots,e_n].$$
Let $p_i$ denote the $i$-th power sum in $\P(n)$ for $i\ge 1$, i.e. $p_i=x_1^i+\cdots+x_n^i$. By the Newton formula
\begin{equation}
  \label{Newton formula}
  ke_k=\sum_{i=0}^k(-1)^{i-1}e_{k-i}p_i
\end{equation}
where we set $e_0=p_0=1$. Then since the ground field $\F$ is of characteristic zero or prime to $n!$, $\P(n)^{\Sigma_n}$ is alternatively expressed as
\begin{equation}\label{power sum}
  \P(n)^{\Sigma_n}=\F[p_1,\ldots,p_n].
\end{equation}
This simple observation is the core idea to construct generators of $\H(m,U(n))^{\Sigma_n}$.

Let $q_i\coloneqq x_1^iy_1+\cdots+x^i_ny_n\in\P(n)\otimes\Lambda(y_1,\ldots,y_n)$ for $i\ge 0$. We observe a property of these element. For non-negative integers $k_1,\ldots,k_n$, let
$$q_{k_1,\ldots,k_n}\coloneqq\sum_{\sigma\in\Sigma_n}\sgn(\sigma)x_{\sigma(1)}^{k_1}\cdots x_{\sigma(n)}^{k_n}.$$
Then $q_{k_1,\ldots,k_n}$ is anti-symmetric. Note that $q_{k_1,\ldots,k_n}$ can be trivial for some $k_1,\ldots,k_n$. For instance,
\begin{align*}
  q_{1,0,0}&=\sgn(123)x_1+\sgn(213)x_2+\sgn(321)x_3\\
  &\quad+\sgn(132)x_1+\sgn(231)x_2+\sgn(312)x_3\\
  &=0
\end{align*}
where $n=3$. Clearly, each anti-symmetric polynomial in $\P(n)$ is a linear combination of $q_{k_1,\ldots,k_n}$. By a straightforward calculation, we can prove the following.

\begin{lem}\label{anti-symmetric}
  For non-negative integers $k_1,\ldots,k_n$,
  $$q_{k_1}\cdots q_{k_n}=q_{k_1,\ldots,k_n}y_1\cdots y_n.$$
  In particular, every anti-symmetric polynomial in $\P(n)$ is generated by $q_i$.
\end{lem}

Let $[m]\coloneqq\{1,2,\ldots,m\}$. We consider an ordering on the power set $2^{[m]}$ which sorts by cardinality first and by the lexicographic ordering with $1<2<\cdots<m$ next.

\begin{exam}\label{example 2^3}
  The power set $2^{[3]}$ is ordered as
  $$\emptyset<\{1\}<\{2\}<\{3\}<\{1,2\}<\{1,3\}<\{2,3\}<\{1,2,3\}.$$
\end{exam}

We give a basis of $\widetilde{\H}(m,U(n))^{\Sigma_n}$. For $I=\{i_1<\cdots<i_k\}\subset[m]$, let
$$y_j^I\coloneqq y_j^{i_1}\cdots y_j^{i_k}$$
and for $d_1,\ldots,d_{2^m-1}\ge 0$, let
$$z(d_1,\ldots,d_{2^m-1})\coloneqq\prod_{k=1}^{2^m-1}\prod_{i=d_1+\cdots+d_{k-1}+1}^{d_1+\cdots+d_k}y^{S_k}_i$$
where $S_k$ is in $2^{[m]}=\{\emptyset<S_1<\cdots<S_{2^m-1}\}$ by the above ordering.

\begin{exam}\label{M}
  For $m=3$, $z(d_1,\ldots,d_7)$ of degree $2$ are:
  \begin{alignat*}{3}
    z(0,1,1,0,0,0,0)&=y_1^2y_2^3\quad&z(1,0,1,0,0,0,0)&=y_1^1y_2^3\quad&z(1,1,0,0,0,0,0)&=y_1^1y_2^2\\
    z(0,0,0,0,0,1,0)&=y_1^2y_1^3&z(0,0,0,0,1,0,0)&=y_1^1y_1^3&z(0,0,0,1,0,0,0)&=y_1^1y_1^2\\
    z(0,0,2,0,0,0,0)&=y_1^3y_2^3&z(0,2,0,0,0,0,0)&=y_1^2y_2^2&z(2,0,0,0,0,0,0)&=y_1^1y_2^1
  \end{alignat*}
\end{exam}

Let $\mathcal{M}(m,n)$ be a subset of $\widetilde{\H}(m,U(n))$ consisting of monic monomials of the form
$$M(x_1,\ldots,x_n)z(d_1,\ldots,d_{2^m-1}).$$
Clearly, for every monic monomial $w$ of $\widetilde{\H}(m,U(n))$, there is a unique $v\in\mathcal{M}(m,n)$ such that
\begin{equation}\label{unique monomial}
  w=\sigma(v)
\end{equation}
for some $\sigma\in\Sigma_n$. Since the ground field $\F$ is of characteristic zero or prime to $n!$, we get the following.

\begin{lem}\label{basis H}
  The set
  $$\left\{\sum_{\sigma\in\Sigma_n}\sigma(v)\;\middle|\; v\in\mathcal{M}(m,n)\right\}$$
  is a basis of $\widetilde{\H}(m,U(n))^{\Sigma_n}$.
\end{lem}

Now we are ready to give a basis of $\H(m,U(n))^{\Sigma_n}$.

\begin{lem}\label{monomial}
  $\H(m,U(n))^{\Sigma_n}$ is spanned by elements represented by
  $$\sum_{\sigma\in\Sigma_n}\sigma(M(x_1,\ldots,x_n)z(d_1,\ldots,d_{2^m-1}))\in\widetilde{\H}(m,U(n))^{\Sigma_n}$$
  where $M(x_1,\ldots,x_n)$ are monomials, satisfying the following conditions:
  \begin{enumerate}
    \item $M(x_1,\ldots,x_n)=M(x_1,\ldots,x_r,0,\ldots,0)$ for $r=d_1+\cdots+d_{2^m-1}$;
    \item if the above representative includes the term
    $$P(x_{d_1+\cdots+d_{i-1}+1},\ldots,x_{d_1+\cdots+d_{i-1}+d_i})Q$$
    where $Q$ does not include  $x_{d_1+\cdots+d_{i-1}+1},\ldots,x_{d_1+\cdots+d_{i-1}+d_i}$, then $P$ is symmetric for $|S_i|$ even and anti-symmetric for $|S_i|$ odd.
  \end{enumerate}
\end{lem}

\begin{proof}
  (1) As in the proof of Theorem \ref{Ramras-Stafa}, the natural map
  $$\widetilde{\H}(m,U(n))^{\Sigma_n}\to\H(m,U(n))^{\Sigma_n}$$
  is surjective. Then by Lemma \ref{basis H}, each element of $\H(m,U(n))^{\Sigma_n}$ is a linear combination of elements of $\widetilde{\H}(m,U(n))^{\Sigma_n}$ represented by elements of Lemma \ref{basis H}. By definition, $M(x_1,\ldots,x_n)$ is invariant under permutations on $\{r+1,r+2,\ldots,n\}$. Then by \eqref{power sum}
  $$M(x_1,\ldots,x_n)=\sum_{i\ge 0}M'_i(x_1,\ldots,x_r)p_i(x_{r+1},\ldots,x_n)$$
  where $M'_i$ are monomials and $p_i$ denotes the $i$-th power sum as above. Since
  $$p_i(x_{r+1},\ldots,x_n)\equiv -p_i(x_1,\ldots,x_r)\quad\text{in}\quad\P(n)_{\Sigma_n}$$
  by \eqref{power sum}, we can choose the monomials $M(x_1,\ldots,x_n)$ as stated.

  \noindent(2) By definition, the representative in the statement is invariant under permutations on $\{d_1+\cdots+d_{i-1}+1,d_1+\cdots+d_{i-1}+2\ldots,d_1+\cdots+d_{i-1}+d_i\}$, so that $z(d_1,\ldots,d_{2^m-1})$ is invariant for $|S_i|$ even and anti-invariant for $|S_i|$ odd. Then the second statement follows.
\end{proof}

We define an ordering on $\mathcal{M}(m,n)$. Note that those elements are bigraded by the degrees in the polynomial and the exterior algebra part.
\begin{enumerate}
  \item We sort by total degree.

  \item After (1), we sort by degree with respect to the exterior algebra part.

  \item After (2), $M(x_1,\ldots,x_n)z(d_1,\ldots,d_{2^m-1})>M'(x_1,\ldots,x_n)z(d_1',\ldots,d_{2^m-1}')$ whenever $d_1+\cdots+d_{2^m-1}<d_1'+\cdots+d_{2^m-1}'$.

  \item After (3), $M(x_1,\ldots,x_n)z(d_1,\ldots,d_{2^m-1})>M'(x_1,\ldots,x_n)z(d_1',\ldots,d_{2^m-1}')$ whenever $d_i=d_i'$ for $i<k$ and $d_k>d_k'$ for some $k$.

  \item Finally, we sort by the lexicographic ordering on $\P(n)$ with $x_1<\cdots<x_n$.
\end{enumerate}

Given any $w\in\widetilde{\H}(m,U(n))^{\Sigma_n}$, it follows from Lemma \ref{basis H} that
$$w=\sum_{v\in\mathcal{M}(m,n)}a_v\sum_{\sigma\in\Sigma_n}\sigma(v)$$
where $a_v\in\F$. We define the least term of $w$ to be the least $v$ such that $a_v\ne 0$, which is well defined by the observation on \eqref{unique monomial}.

\begin{exam}
  By Examples \ref{example 2^3} and \ref{M}, elements of $\mathcal{M}(3,2)$ of degree 2 are ordered as
  $$y_1^3y_2^3<y_1^2y_2^3<y_1^2y_2^2<y_1^1y_2^3<y_1^1y_2^2<y_1^1y_2^1<y_1^2y_1^3<y_1^1y_1^3<y_1^1y_1^2<x_1<x_2.$$
\end{exam}

By analogy to the power sums in $\P(n)$, we define for $\emptyset\ne I\subset[m]$,
$$z(d,I)\coloneqq x^{d-1}_1y_1^I+\cdots+x_n^{d-1}y_n^I.$$
Clearly, $z(d,I)\in\H(m,U(n))^{\Sigma_n}$. Clearly, if $k>n$, then for any $d_1,\ldots,d_k\ge 1$ and $\emptyset\ne I_1,\ldots,I_k\subset[m]$,
$$z(d_1,I_1)\cdots z(d_k,I_k)=0.$$

\begin{lem}\label{least term}
  The least term of $z(d_1,I_1)\cdots z(d_k,I_k)$ for $I_1\leq\cdots\leq I_k$ and $k\le n$ is
  \[
    x_1^{d_1-1}\cdots x_k^{d_k-1}y_1^{I_1}\cdots y_k^{I_k}.
  \]
\end{lem}

\begin{proof}
  The exterior algebra part of any monomial in $z\coloneqq z(d_1,I_1)\cdots z(d_k,I_k)$ is
  $$y_{i_1}^{J_1}\cdots y_{i_l}^{J_l}$$
  where $i_1<\cdots<i_l$ and $J_i=I_{j_1}\cup\cdots\cup I_{j_{k_i}}$ for disjoint $I_{j_1},\ldots,I_{j_{k_i}}$ such that distinct $J_i$ include no common $I_j$. Then by the condition (3) in the definition of the ordering on $\mathcal{M}$, the exterior algebra part of the least term is $y_1^{I_1}\cdots y_k^{I_k}$. Thus the least term of $z(d_1,I_1)\cdots z(d_k,I_k)$ is $x_1^{d_1-1}\cdots x_k^{d_k-1}y_1^{I_1}\cdots y_k^{I_k}$, as stated.
\end{proof}

We set
$$\overline{\S}(m,U(n))\coloneqq \{z(d,I)\mid d\ge 1\text{ and }\emptyset\ne I\subset[m]\}.$$
If $A$ is a subset of $\widetilde{\H}(m,U(n))$, then we will use the same symbol $A$ for the image of $A$ under the projection $\widetilde{\H}(m,U(n))\to\H(m,U(n))$.

\begin{thm}\label{generator U not minimal}
  $\H(m,U(n))^{\Sigma_n}$ is generated by $\overline{\S}(m,U(n))$.
\end{thm}

\begin{proof}
  Let $R$ denote the subring of $\H(m,U(n))$ generated by $\overline{\S}(m,U(n))$. Take any element $w\in\H(m,U(n))^{\Sigma_n}$ represented by an element $\widetilde{w}\in\widetilde{\H}(m,U(n))^{\Sigma_n}$ in Lemma \ref{monomial}. By \eqref{power sum} and Lemmas \ref{anti-symmetric} and \ref{least term}, the least term of $\widetilde{w}$ is included in some element $v\in R$. Since $v$ and $\widetilde{w}$ are $\Sigma_n$-invariant, their least terms coincide, and so by eliminating the least terms inductively, we obtain $w\in R$. Thus the statement holds.
\end{proof}


\subsection{Low dimensional cohomology for unitary groups}

We continue to assume that the ground field $\F$ is of characteristic zero or prime to $n!$. We determine $H^*(\Hom(\Z^m,U(n))_1;\F)\cong\H(m,U(n))^{\Sigma_n}$ in low dimensions, and show that a certain subset of $\overline{\S}(m,U(n))$ must be contained in every subset of $\overline{\S}(m,U(n))$ that generates $\H(m,U(n))^{\Sigma_n}$.

\begin{lem}\label{base}
  For any $\sigma\in\Sigma_n$,
  $$\mathcal{B}(\sigma)\coloneqq\{x_{\sigma(1)}^{i_1}x_{\sigma(2)}^{i_2}\cdots x_{\sigma(n-1)}^{i_{n-1}}\in\P(n)_{\Sigma_n}\mid i_k\leq n-k\text{ for }1\le k\le n-1\}$$
  is a basis of $\P(n)_{\Sigma_n}$.
\end{lem}

\begin{proof}
  As in \cite[Section 10, Proposition 3]{F}, $\mathcal{B}(1)$ is a basis of $\P(n)_{\Sigma_n}$. Then the statement holds by symmetry of $\P(n)_{\Sigma_n}$.
\end{proof}

Let $\mathcal{E}(m,n)$ denote the set of all monic monomials of $\bigotimes_{i=1}^m\Lambda(y_1^i,y_2^i,\ldots,y_n^i)$. Then for any map $\alpha\colon\mathcal{E}(m,n)\to\Sigma_n$, the set
$$\overline{\mathcal{B}}(\alpha)\coloneqq\{XY\mid X\in\mathcal{B}(\alpha(Y)),\;Y\in\mathcal{E}(m,n)\}$$
is a basis of $\H(m,U(n))$.

Let $\F\langle S\rangle$ denote the free commutative graded algebra over $\F$ generated by a graded set $S$, and define a subspace of $\F\langle\overline{\S}(m,U(n))\rangle$
$$V(m,n)\coloneqq\{z(d_1,I_1)\cdots z(d_k,I_k)\mid I_1<\cdots<I_k\text{ and }k+\max\{d_1,\ldots,d_k\}-1\le n\}.$$

\begin{lem}
  \label{inj}
  The projection
  $$\F\langle\overline{\S}(m,U(n))\rangle\to\H(m,U(n))^{\Sigma_n}$$
  is injective on $V(m,n)$
\end{lem}

\begin{proof}
  By Lemma \ref{least term}, the least term of $z(d_1,I_1)\cdots z(d_k,I_k)\in V(m,n)$ for $I_1<\cdots<I_k$ is
  $$x_1^{d_1-1}\cdots x_k^{d_k-1}y_1^{I_1}\cdots y_k^{I_k}.$$
  Note that $d_i-1\le\max\{d_1,\ldots,d_k\}-1\le n-k\le n-i$ for $1\le i\le k$. Then the least terms of elements of $V(m,n)$ are mutually distinct elements of $\overline{\mathcal{B}}(\alpha)$ for any map $\alpha\colon\mathcal{E}(m,n)\to\Sigma_n$ satisfying $\alpha(y_1^{I_1}\cdots y_k^{I_k})\in\Sigma_k$. Thus elements of $V(m,n)$ are linearly independent in $\H(m,U(n))^{\Sigma_n}$, completing the proof.
\end{proof}

We set
$$\S(m,U(n))\coloneqq\{z(d,I) \in\overline{\S}(m,U(n))\mid d+|I|-1\leq n\}.$$
Now we are ready to describe $\H(m,\Sigma_n)^{\Sigma_n}$ in low dimensions.

\begin{thm}[Theorem \ref{low dim intro}]\label{low dim}
  The map $\F\langle\S(m,U(n))\rangle\to\H(m,U(n))^{\Sigma_n}$ is an isomorphism in dimension $\le 2n-m$.
\end{thm}

\begin{proof}
  By a degree reason, the inclusion $\F\langle\S(m,U(n))\rangle\to\F\langle\overline{\S}(m,U(n))\rangle$ is an isomorphism in dimension $\le 2n-m$, and the image of this inclusion is contained in $V(m,n)$ in dimension $\le 2n-m$. Then by Lemmas \ref{generator U not minimal} and \ref{inj}, the proof is complete.
\end{proof}

We end this subsection by proving the following lemma.

\begin{lem}\label{stable range U}
  $\S(m,U(n))$ is included in every subset of $\overline{\S}(m,U(n))$ that generates $\H(m,U(n))^{\Sigma_n}$.
\end{lem}

\begin{proof}
  Let $F_i$ denote the subspace of $\H(m,U(n))^{\Sigma_n}$ generated by $z(d_1,I_1)\cdots z(d_k,I_k)$ such that $d_1+\cdots+d_k+|I_1|+\cdots+|I_k|-k\le i$. Then we get a filtration
  $$\{0\}=F_0\subset F_1\subset F_2\subset\cdots\subset\H(m,U(n))^{\Sigma_n}.$$
  Note that this filtration is induced from that of $\widetilde{\H}(m,U(n))^{\Sigma_n}$ defined by word-length with respect to $x_i$ and $y_i^j$. Clearly,
  $$F_i\cdot F_j\subset F_{i+j}\quad\text{and}\quad\S(m,U(n))\subset V(m,n)\cap F_n.$$
  Suppose that an element $z(d,I)$ of $\S(m,U(n))$ with $d+|I|-1=i$ is expressed by the remaining elements of $\S(m,U(n))$ Then $z(d,I)$ belongs to $F_{i-1}\cap V(m,n)$. This is a contradiction to Lemma \ref{inj}, completing the proof.
\end{proof}


\subsection{Minimal generating set for unitary groups}

We continue to assume that the ground field $\F$ is of characteristic zero or prime to $n!$. We give a minimal generating set of $\H(m,U(n))^{\Sigma_n}$ contained in $\overline{\S}(m,U(n))$. We begin with a simple observation. In $\H(3,U(2))^{\Sigma_2}$,
\begin{align*}
  &z(1,\{1\})z(1,\{2\})z(1,\{3\})\\
  &=(y_1^1+y_2^1)(y_1^2+y_2^2)(y_1^3+y_2^3)\\
  &=(y_1^1y_1^2+y_2^1y_2^2)(y_1^3+y_2^3)-(y_1^1y_1^3+y_2^1y_2^3)(y_1^2+y_2^2)\\
  &\quad+(y_1^2y_1^3+y_2^2y_2^3)(y_1^1+y_2^1)-2(y_1^1y_1^2y_1^3+y_2^1y_2^2y_2^3)\\
  &=z(1,\{1,2\})z(1,\{3\})-z(1,\{1,3\})z(1,\{2\})\\
  &\quad+z(1,\{2,3\})z(1,\{1\})-2z(1,\{1,2,3\})
\end{align*}
Then we get a relation among $\overline{\S}(3,U(2))$. We can think that this relation comes out from the Newton formula \eqref{Newton formula}, hence we prove the following generalized Newton formula to deduce relations among $\overline{\S}(m,U(n))$.

For $d\ge 0$ and $I=\{i_1<\dots<i_k \}\subset[ m ]$, we define an element of $\widetilde{\H}(m,U(n))$
$$e(d,I)\coloneqq \sum_{\sigma\in\Emb(k,n)} x_{\sigma(1)}^dy_{\sigma(1)}^{i_1}y_{\sigma(2)}^{i_2}\dots y_{\sigma(k)}^{i_{k}}.$$

\begin{lem}\label{Newton}
  For $k\le n$ and $I=\{i_1<\dots <i_k\}\subset[m]$,
  \begin{align*}
    &\sum_{l=0}^{k-2}\sum_{\substack{i_1\in J\subset I\\|J|=l+1}}(-1)^{l+d(J)}l!z(d,J)e(0,I-J)+(-1)^{k+1}(k-1)!z(d,I)\\
    &=
    \begin{cases}
      e(d-1,I)&(k\le n)\\
      0&(k>n)
    \end{cases}
  \end{align*}
  in $\widetilde{\H}(m,U(n))$, where $d(J)=(\sum_{i_j\in J} j)-\frac{|J|(|J|+1)}{2}$.
\end{lem}

\begin{proof}
  Suppose that the identities
  \begin{align*}
    &\sum_{l=0}^{k-2}\sum_{\substack{i_1\in J\subset I\\|J|=l+1}}(-1)^{l+d(J)}l!x_{i}^{d-1}y_i^Je(0,I-J)+(-1)^{k}(k-1)!x_{i}^{d-1}y_{i}^I\\
    &=\begin{cases}
      \displaystyle\sum_{\sigma}x_{i}^{d-1}y_{i}^{i_1}y_{\sigma(2)}^{i_2}\dots y_{\sigma(k)}^{i_k}&(k\le n)\\
      0&(k>n).
    \end{cases}
  \end{align*}
  in $\widetilde{\H}(m,U(n))$ hold, where $\sigma$ ranges over all injections $[k]-1\to[n]-i$. Then the identities in the statement are obtained by summing up these identities. Hence we prove these identities. Let
  $$f(i,I)\coloneqq \sum_{\substack{\sigma\in\Emb(k,n)\\i\in\sigma([k])}}y_{\sigma(1)}^{i_1}\cdots y_{\sigma(k)}^{i_{k}}.$$
  Then for $1\leq l\leq k-2$,
  \begin{align*}
    &\sum_{\substack{i_1\in J\subset I\\|J|=l}}(-1)^{d(J)}x_{i}^{d-1}y_i^Jf(i,I-J)\\
    &=\sum_{\substack{i_1\in J\subset I\\|J|=l\\ I-J=\{j_1,\dots j_{k-l}\}}}(-1)^{d(J)}x_{i}^{d-1}y_i^J\sum_{\substack{\sigma\in\Emb(k-l,n)\\i\in\sigma([k-l])}}y_{\sigma(1)}^{j_1}\cdots y_{\sigma(k-l)}^{j_{k-l}}.
  \end{align*}
  Moreover,
  \begin{align*}
    &\sum_{\substack{\sigma\in\Emb(k-l,n)\\i\in\sigma([k-l])}}y_{\sigma(1)}^{j_1}\cdots y_{\sigma(k-l)}^{j_{k-l}}\\
    &=\sum_{1\leq s\leq k-l}(-1)^{s-1}y_i^{j_s}\sum_{\substack{\sigma\in\Emb(k-l-1,n)\\i\notin\sigma([k-l-1])}}y_{\sigma(1)}^{j_1}\cdots y_{\sigma(s-1)}^{j_{s-1}}y_{\sigma(s)}^{j_{s+1}} \cdots  y_{\sigma(k-l-1)}^{j_{k-l}}.
  \end{align*}
  Then it follows that
  \begin{align*}
  &\sum_{\substack{i_1\in J\subset I\\|J|=l}}(-1)^{d(J)}x_{i}^{d-1}y_i^Jf(i,I-J)\\
  &=\sum_{\substack{i_1\in J\subset I\\|J|=l\\ I-J=\{j_1,\dots j_{k-l}\}}}\sum_{1\leq s\leq k-l}(-1)^{d(J\bigcup \{j_s\})}x_{i}^{d-1}y_i^{J\bigcup\{j_s\}}\\
  &\qquad\qquad\qquad\sum_{\substack{\sigma\in\Emb(k-l-1,n)\\i\notin\sigma([k-l-1])}}y_{\sigma(1)}^{j_1}\cdots y_{\sigma(s-1)}^{j_{s-1}}y_{\sigma(s)}^{j_{s+1}} \cdots  y_{\sigma(k-l-1)}^{j_{k-l}}\\
  &=\sum_{\substack{i_1\in J\subset I\\|J|=l+1}}(-1)^{d(J)}lx_{i}^{d-1}y_i^J(e(0,I- J)-f(i,I- J)).
  \end{align*}
  By a similar calculation, one also gets
  $$\sum_{\substack{i_1\in J\subset I \\ |J|=k-1}}(-1)^{d(J)}x_{i}^{d-1}y_i^Jf(i,I-J)=(k-1)x_{i}^{d-1}y_{i}^I$$
  and for $k\le n$,
  $$\sum_{\substack{i_1\in J\subset I \\ |J|=1}}(-1)^{d(J)}x_{i}^{d-1}y_i^J(e(0,I-J)-f(i,I-J))=\sum_{ \sigma }x_{i}^{d-1}y_{i}^{i_1}y_{\sigma(2)}^{i_2}\cdots y_{\sigma(k)}^{i_{k}},$$
  where $\sigma$ ranges over all injections $[k]-1\to[n]-i$. Thus by combining these three identities, we obtain the desired identities.
\end{proof}

We need the following three lemmas to deduce relations among $\overline{\S}(m,U(n))$ from Lemma \ref{Newton}.

\begin{lem}\label{coinv}
  For $k\le n$,
  $$\sum_{i_1+\cdots +i_k=n-k+1}x_1^{i_1}\cdots x_k^{i_k}=0\quad\text{in}\quad\P(n)_{\Sigma_n}.$$
\end{lem}

\begin{proof}
  Define $h_{a,k},\,e_{a,b,k}\in\P(n)$ by the identities of formal power series in $t$ over $\P(n)$
  $$\sum_{a\ge 0} h_{a,k}t^a=\prod_{1\leq j\leq k}\frac {1}{1-x_jt}\quad\text{and}\quad\sum_{0\leq a\leq k} (-1)^ae_{a,b}t^a=\prod_{k\leq j\leq k+b}(1-x_jt).$$
  By a straightforward calculation, we can see that
  $$h_{a,k}=\sum_{i_1+\cdots+i_k=a}x_1^{i_1}\cdots x_k^{i_k}\quad\text{and}\quad e_{a,b,k}=\sum_{k\leq i_1 < \cdots <i_a\leq k+b }x_{i_1}\cdots x_{i_a}.$$
  Note that
  \begin{align*}
   &\prod_{1\leq j\leq k}\frac {1}{1-x_jt}-\prod_{k+1\leq j\leq n}\left(1-x_jt\right)\\
   &=\frac{1-\prod_{1\leq j\leq n}\left(1-x_jt\right)}{\prod_{1\leq j\leq k}(1-x_jt)}\\
   &=\left(\sum_{1\leq a \leq n}(-1)^{l+1} e_{a,n-1,1}t^a\right)\prod_{1\leq j\leq k}\frac {1}{1-x_jt}.
  \end{align*}
  Since $e_{a,n-1,1}=0$ in $\P(n)_{\Sigma_n}$, we get
  $$\prod_{1\leq l\leq k}\frac {1}{1-x_jt}=\prod_{1\leq j\leq n-k}(1-x_{k+j}t)$$
  as formal power series over $\P(n)_{\Sigma_n}$. Note that the coefficient of $t^{n-k+1}$ in the LHS is $h_{n-k+1,k}$ and the RHS is 0. Thus $h_{n-k+1,k}=0$ in $\P(n)_{\Sigma_n}$, which is the identity in the statement.
\end{proof}

Let $\genfrac{\{}{\}}{0pt}{0}{n}{k}$ denote the Stirling number of the second count. Then $\genfrac{\{}{\}}{0pt}{0}{n}{k}$ counts the number of partitions of $[n]$ into $k$ non-empty subsets. Recall from \cite[(1.94d)]{S} that the following identity holds.

\begin{lem}
  \label{Stirling 2}
  There is a equation
  $$\sum_{k=1}^n(-1)^k(k-1)!\genfrac{\{}{\}}{0pt}{0}{n}{k}=0.$$
\end{lem}

\begin{lem}\label{least}
  For $k\ge 2$, let $I_1,\ldots I_k$ be pairwise disjoint subsets of $[n]$, and let $d_1,\ldots d_k$ be non-negative integers. Then the element
  $$w\coloneqq\sum_{\sigma\in\Emb(k,n)}x_{\sigma(1)}^{d_1}\cdots x_{\sigma(k)}^{d_k}y_{\sigma(1)}^{I_1}\cdots y_{\sigma(k)}^{I_k}$$
  of $\widetilde{\H}(m,U(n))$ includes the term
  $$(-1)^{k+1+d(I_1,\ldots I_k)}(k-1)!z(1+d_1+\cdots+d_k,I_1\cup\cdots\cup I_k)$$
  where $d(I_1,\ldots I_k)\in\Z/2$ is defined by
  $$(-1)^{d(I_1,\ldots I_k)}y_{1}^{I_1\cup\cdots\cup I_k}=y_{1}^{I_1}\cdots y_{1}^{I_k}.$$
\end{lem}

\begin{proof}
  We prove the statement by induction on $k$. For $k=2$, it is easy to see that
  \begin{align*}
    &\sum_{\sigma\in\Emb(2,n)}x_{\sigma(1)}^{d_1}x_{\sigma(2)}^{d_2}y_{\sigma(1)}^{I_1}y_{\sigma(2)}^{I_2}\\
    &=(-1)^{d(I_1,I_2)+1}z(d_1+d_2+1, I_1 \cup I_2)+z(d_1+1,I_1)z(d_2+1,I_2).
  \end{align*}
  Then the statement holds.

  Assume that the statement holds for $<k$. Let
  $$r_i=\sum_{j\in J_i}d_j\quad\text{and}\quad K_i=\bigcup_{j\in J_i}I_j.$$
  Then by a straightforward calculation, we can see that
  \begin{align*}
    &z(1+d_1,I_1)\cdots z(1+d_k,I_k)-\sum_{\sigma\in\Emb(k,n)}x_{\sigma(1)}^{d_1}\cdots x_{\sigma(k)}^{d_k}y_{\sigma(1)}^{I_1}\cdots y_{\sigma(k)}^{I_k}\\
    &=\sum_{s=1}^{k-1}\sum_{\substack{J_1\sqcup\cdots\sqcup J_s=[k] \\ J_1,\ldots,J_s\ne\emptyset}}\sum_\sigma(-1)^{d'(J_1, \ldots J_s)}x_{\sigma(1)}^{r_1}\cdots x_{\sigma(s)}^{r_s}y_{\sigma(1)}^{K_1}\cdots y_{\sigma(s)}^{K_s}
  \end{align*}
  where $d'(J_1, \ldots J_s)\in\Z/2$ is defined by
  $$(-1)^{d'(J_1, \dots J_s)}y_{1}^{K_1}\cdots y_{1}^{K_s}=y_{1}^{I_1}\cdots y_{1}^{I_k}.$$
  By assumption, the coefficient of $z(1+d_1+\cdots+d_k,I_1\cup\cdots\cup I_k)$ in $w$ is
  $$\sum_{s=1}^{k-1}(-1)^{d(I_1, \dots I_k)+s+1}(s-1)!\genfrac{\{}{\}}{0pt}{0}{k}{s}.$$
  Thus the induction is complete by Lemma \ref{Stirling 2}.
\end{proof}

Now we are ready to prove:

\begin{thm}[Theorem \ref{generator intro}]\label{minimal generator U}
  $\H(m,U(n))^{\Sigma_{n}}$ is minimally generated by $\S(m,U(n))$.
\end{thm}

\begin{proof}
  By Theorem \ref{generator U not minimal} and Lemma \ref{stable range U}, it is sufficient to show that $z(d,I)\in\overline{\S}(m,U(n))$ for $d+|I|\ge n+2$ belongs to a subring of $\H(m,U(n))^{\Sigma_n}$ generated by $\S(m,U(n))$. First, we consider the case $|I|>n$. Applying Lemma \ref{Newton} for $k=1$, one gets that for any $J\subset[m]$, $e(0,J)$ belongs to the subring of $\H(m,U(n))^{\Sigma_n}$ generated by $z(d,I)\in\overline{\S}(m,U(n))$ for $|I|\le n$. One also gets that $z(d,I)\in\S(m,n)$ for $|I|>n$ belongs to the subring of $\H(m,U(n))^{\Sigma_n}$ generated by $e(0,J)$ for $J\subset[m]$ and $z(d,J)\in\S(m,n)$ for $|J|<|I|$. Thus $z(d,I)\in\overline{\S}(m,U(n))$ for $|I|>n$ turn out to belong to the subring of $\H(m,U(n))^{\Sigma_n}$ generated by $z(d,I)\in\overline{\S}(m,U(n))$ for $|I|\le n$.

  Next, we consider the case $|I|\le n$ and $d+|I|\ge n+2$. Let $I=\{i_1<\cdots<i_k\}$ be a subset of $[n]$, and let $d$ be a positive integer with $d+k\ge n+2$. By Lemma \ref{coinv}, there is an identity in $\H(m,U(n))^{\Sigma_n}$
  $$\sum_{j_1+\cdots+j_k=n-k+1}\sum_{\sigma\in\Emb(k,n)} x_{\sigma(1)}^{d-n+k-2}x_{\sigma(1)}^{j_1}\cdots x_{\sigma(1)}^{j_k}y_{\sigma(1)}^{i_1}\cdots y_{\sigma(k)}^{i_k}=0.$$
  On the other hand, it follows from Lemma \ref{least} that the LHS includes a non-zero multiple of $z(d,I)$. Thus $z(d,I)$ turns out to belong to the subring of $\H(m,U(n))^{\Sigma_n}$ generated by $z(d',I')\in\overline{\S}(m,U(n))$ with $d'+|I'|<d+|I|$. Therefore the proof is complete by induction on $d+|I|$.
\end{proof}

Let
$$\S(m,SU(n))\coloneqq\{z(d,I) \in\S(m,U(n))\mid d>1\text{ or }|I|>1\}.$$
Then $\S(m,SU(n))$ is a subset of $\S(m,U(n))$ consisting of elements of degree $>1$. Thus by \eqref{U-SU} and Theorem \ref{minimal generator U}, we obtain the following corollary.

\begin{cor}
  \label{minimal generator SU}
  $\H(m,SU(n))^{\Sigma_{n}}$ is minimally generated by $\S(m,SU(n))$.
\end{cor}

We give example calculations of the cohomology of $\Hom(\Z^2,SU(n))_1$ for $n=2,3$.

\begin{exam}\label{SU(2)}
  By Example \ref{example Poincare series}, $\H^*(\Hom(\Z^2,SU(2))_1;\F)$ is non-trivial only for $*=0,2,3$, where $\F$ is a field of characteristic zero or prime to 2. Then all products in $\H^*(\Hom(\Z^2,SU(2))_1;\F)$ are trivial. By Theorem \ref{minimal generator SU}, $\H^*(\Hom(\Z^2,SU(2))_1;\F)$ is generated by
  $$a\coloneqq z(2,\{1\}),\quad b\coloneqq z(2,\{1\}),\quad c\coloneqq z(1,\{1,2\}).$$
  Thus we obtain
  $$\H^*(\Hom(\Z^2,SU(2))_1;\F)=\F\langle a,b,c\rangle/(a,b,c)^2$$
  where $|a|=|b|=3$ and $|c|=2$.
\end{exam}

\begin{exam}\label{SU(3)}
  We assume the ground field $\F$ is of characteristic zero or prime to 6. By Theorem \ref{minimal generator SU}, $\H^*(\Hom(\Z^2,SU(3))_1;\F)$ is generated by
  $$a_i\coloneqq z(i+1,\{1\}),\quad b_i\coloneqq z(i+1,\{1\}),\quad c_i\coloneqq z(i,\{1,2\})$$
  for $i=1,2$. Then $(a_1,a_2,b_1,b_2,c_1,c_2)^3=0$. Moreover, by a direct computation, we can verify that the following are trivial in $\H(2,SU(n))$.
  $$c_1c_2,\quad a_2c_2,\quad b_2c_2,\quad a_2c_1-a_1c_2,\quad b_2c_1-b_1c_2,\quad a_2b_1-b_2a_1.$$
  Let $I$ be the ideal of $\F\langle a_1,a_2,b_1,b_2,c_1,c_2\rangle$ generated by these elements. It is easy to see that the Poincar\'e series of
  $$\F\langle a_1,a_2,b_1,b_2,c_1,c_2\rangle/(a_1,a_2,b_1,b_2,c_1,c_2)^3+I$$
  coincides with that of $\Hom(\Z^2,SU(3))_1$ in Example \ref{example Poincare series}. Thus we obtain
  $$\H^*(\Hom(\Z^2,SU(3))_1;\F)\cong\F\langle a_1,a_2,b_1,b_2,c_1,c_2\rangle/(a_1,a_2,b_1,b_2,c_1,c_2)^3+I.$$
\end{exam}


\subsection{Results for symplectic groups}

By calculations similar to the case of unitary groups, we give a minimal generating set of $\H(m,Sp(n))^{B_n}$. Since the calculations are quite similar, we only state results without proofs.

We assume that the ground field $\F$ is of characteristic zero or prime to $2(n!)$. Then analogously to \eqref{power sum}, we have
$$\P(n)^{B_n}=\F[p_2,p_4,\ldots,p_{2n}].$$
For $d_1,\ldots,d_{2^m-1}\ge 1$, let
$$w(d_1,\ldots,d_{2^m-1})\coloneqq\prod_{k=1}^{2^m-1}\prod_{i=d_1+\cdots+d_{k-1}+1}^{d_1+\cdots+d_k}x_i^{\epsilon(|I_k|)}y_i^{I_k}$$
where
$$\epsilon(k)=\begin{cases}0&(k\text{ even})\\1&(k\text{ odd})\end{cases}$$
and $\emptyset\ne I_1<\cdots<I_{2^m-1}$ is the ordering of $2^{[m]}$ defined above. Let $\mathcal{N}(m,n)$ be the subset of $\widetilde{\H}(m,Sp(n))$ consisting of monic monomials of the form
$$M(x_1,\ldots,x_n)w(d_1,\ldots,d_{2^m-1}).$$
Analogously to Lemmas \ref{basis H} and \ref{monomial}, one can prove the following two lemmas.

\begin{lem}\label{basis H Sp}
  The set
  $$\left\{\sum_{\sigma\in B_n}\sigma(v)\,\middle|\,v\in\mathcal{N}(m,n)\right\}$$
  is a basis of $\widetilde{\H}(m,Sp(n))^{B_n}$.
\end{lem}

\begin{lem}
  \label{monomial Sp}
  $\H(m,Sp(n))^{B_n}$ is spanned by elements represented by
  $$\sum_{\sigma\in B_n}\sigma(M(x_1^2,\ldots,x_n^2)w(d_1,\ldots,d_{2^m-1}))$$
  satisfying the following conditions:
  \begin{enumerate}
    \item $M(x_1^2,\ldots,x_n^2)=M(x_1^2,\ldots,x_r^2,0,\ldots,0)$ for $r=d_1+\cdots+d_{2^m-1}$;
    \item if the above representative includes the term
    $$P(x_{d_1+\cdots+d_{i-1}+1}^2,\ldots,x_{d_1+\cdots+d_{i-1}+d_i}^2)Q$$
    where $Q$ does not include  $x_{d_1+\cdots+d_{i-1}+1}^2,\ldots,x_{d_1+\cdots+d_{i-1}+d_i}^2$, then $P$ is symmetric for $|S_i|$ even and anti-symmetric for $|S_i|$ odd.
  \end{enumerate}
\end{lem}

For $\emptyset\ne I\subset[m]$ and $d\ge 1$, let
$$w(d,I)\coloneqq x_1^{2d+\epsilon(|I|)-2}y_1^I+\cdots+x_n^{2d+\epsilon(|I|)-2}y_n^I$$
and let
$$\overline{\S}(m,Sp(n))\coloneqq\{w(d,I)\mid d\ge 1\text{ and }\emptyset\ne I\subset[m]\}.$$
We can define an ordering on $\mathcal{N}(m,n)$ analogously to that on $\mathcal{M}(m,n)$. Then by arguing as in the proof of Theorem \ref{generator U not minimal}, one gets:

\begin{thm}
  \label{generator not minimal Sp}
  $\H(m,Sp)^{B_n}$ is generated by $\overline{\S}(m,Sp(n))$.
\end{thm}

Then it remains to find a minimal generating set of $\H(m,Sp(n))^{B_n}$ contained in $\overline{\S}(m,Sp(n))$. Quite similarly to Lemma \ref{base} (cf. \cite[Section 10, Proposition 3]{F}), we can prove the following.

\begin{lem}
  \label{basis H Sp}
  $\P(n)_{B_n}$ has a base
  $$\{x_1^{2i_1+\epsilon_1}\cdots x_{n-1}^{2i_{n-1}+\epsilon_{n-1}}x_{n}^{\epsilon_{n}}\mid i_k\le n-k\text{ and }\epsilon_k=0,1\}.$$
\end{lem}

We define a subset of $\overline{\S}(m,Sp(n))$ by
$$\S(m,Sp(n))\coloneqq\{w(d,I) \in\overline{\S}(m,Sp(n))\mid 2d+|I|+\epsilon(|I|)-2\leq 2n\}.$$
Similarly to Theorem \ref{low dim}, we can prove the following.

\begin{thm}
  \label{low dim Sp}
    The map $\F\langle\S(m,Sp(n))\rangle\to\H(m,Sp(n))^{B_n}$ is an isomorphism in dimension $\le 2n+1$.
\end{thm}

Let $\S(m,SO(2n+1))\coloneqq\S(m,Sp(n))$. Then by Theorems \ref{cohomology SO Sp} and \ref{low dim Sp}, we get:

\begin{cor}
  \label{low dim SO}
  The map $\F\langle\S(m,SO(2n+1))\rangle\to\H(m,SO(2n+1))^{B_n}$ is an isomorphism in dimension $\le 2n+1$.
\end{cor}

We define
$$W(m,n)\coloneqq\{w(d_1,I_1)\cdots w(d_k,I_k)\mid I_1<\cdots<I_k\text{ and }k+\max\{d_1,\ldots,d_k\}-1\le n\}.$$
By the same argument as Lemma \ref{stable range U}, we can prove:

\begin{lem}\label{stable range Sp}
  $\S(m,Sp(n))$ is included in every subset of $\overline{\S}(m,Sp(n))$ that generates $\H(m,Sp(n))^{B_n}$.
\end{lem}

Thus by using Lemmas \ref{Newton}, similarly to Theorem \ref{minimal generator U}, we finally obtain the following.

\begin{thm}
  \label{minimal generator Sp}
  $\H(m,Sp(n))^{B_n}$ is minimally generated by $\S(m,Sp(n))$.
\end{thm}

The following is immediate from Theorems \ref{cohomology SO Sp} and \ref{minimal generator Sp}.

\begin{cor}
  \label{minimal generator SO}
  $\H(m,SO(2n+1))^{B_n}$ is minimally generated by $\S(m,SO(2n+1))$.
\end{cor}


\section{Homological stability}\label{Homological stability}

Throughout this section, let $G_n$ be either $U(n),SU(n),Sp(n)$ or $SO(2n+1)$. Then there is a sequence of spaces
$$\Hom(\Z^m,G_1)_1\to\Hom(\Z^m,G_2)_1\to\Hom(\Z^m,G_3)_1\to\cdots.$$
In this section, we prove that this sequence satisfies homological stability with rational coefficients and give the best possible stable range. For $d\ge 1$ and $1\le k\le n$, we define a set
$$U(d,k)\coloneqq\{z(d_1,I_1)\cdots z(d_k,I_k)\mid \max\{d_1,\ldots,d_k\}=d\}.$$

\begin{lem}
  \label{dim 2n-m+1}
  If $m+1\le k\le n$ and $d\ge 2$, then least degree elements of $U(d,k)$ are
  $$a_{d,k}(i,I_1,\ldots,I_k)\coloneqq z(d,\{i\})\prod_{j=1}^mz(1,\{j\})\prod_{l=1}^{k-m-1}z(1,I_l)$$
  where $i\in[m]$ and $|I_l|=2$ for all $l$. Otherwise, $U(d,k)=\{0\}$ or least degree elements of $U(d,k)$ are of degree $>2d+2k-m-3$.
%
%
\end{lem}

\begin{proof}
  Let $z\coloneqq z(d_1,I_1)\cdots z(d_k,I_k)$ be a least degree element of $U(d,k)$. We may assume that $d_1=d$. Since $z$ is of degree $2d_1+\cdots+2d_k-2k+|I_1|+\cdots+|I_k|$, it follows from minimality of the degree of $z$ that $d_2=\cdots=d_k=1$ and $|I_1|=1$. Thus the proof is done by an easy degree counting.
\end{proof}

By Lemma \ref{dim 2n-m+1}, elements of $\bigoplus_{d+k=n+2}U(d,k)$ are of degree $\ge 2n-m+1$. Let $U$ denote the vector space spanned by elements of $\bigoplus_{d+k=n+2}U(d,k)$ of degree $2n-m+1$. Then $U=\{0\}$ for $n=m$, and $U$ is spanned by $a_{d,k}(i,I_1,\ldots,I_k)$ in Lemma \ref{dim 2n-m+1} for $n>m$.

For a graded vector space $V$, let $V_d$ denote the $d$-dimensional part of $V$.

\begin{prop}
  \label{V+U}
  For $n\ge m$, there is an isomorphism
  $$\H(m,U(n))^{\Sigma_n}_{2n-m+1}\cong V(m,n)_{2n-m+1}\oplus U.$$
\end{prop}

\begin{proof}
  By Lemma \ref{dim 2n-m+1}, there is a natural surjection
  $$V(m,n)_{2n-m+1}\oplus U\to\H(m,U(n))^{\Sigma_n}_{2n-m+1}.$$
  By Lemma \ref{inj}, if $a_{d,k}(i,I_1,\ldots,I_k)$ and the canonical basis of $V(m,n)_{2n-m+1}$ are linearly independent in $\H(m,U(n))^{\Sigma_n}_{2n-m+1}$ for $d+k=n+2$ and $d\ge 2$, then this map is injective.

  Let $\alpha\colon\mathcal{E}(m,n)\to\Sigma_n$ be a map sending $y_1^1\cdots y_m^my_{m+1}^my_{m+2}^{I_1}\cdots y_{m+k+1}^{I_k}$ to the transposition of $(1\;m+1)$ and other monomials to 1, where $|I_1|=\cdots=|I_k|=2$. We may assume $I_1<\cdots<I_k$. Then by Lemma \ref{least term}, the least term of $a_{d,k}(i,I_1,\ldots,I_k)$ is given by
  $$x_i^{n-m-k}y_1^1\cdots y_i^iy_{i+1}^i\cdots y_{m+1}^my_{m+2}^{I_1}\cdots y_{m+k+1}^{I_k}.$$
  On the other hand, the least terms of the canonical basis are given in the proof of Lemma \ref{inj}. Then since the exponents of $x_i$ in the least terms of the canonical basis of $V(m,n)_{2n-m+1}$ are less than $n-m$, the least terms of $a_{d,k}(i,I_1,\ldots,I_k)$ and the canonical basis of $V(m,n)_{2n-m+1}$ are mutually distinct elements of $\overline{\mathcal{B}}(\alpha)$, where we consider the map $\alpha$ because $x_{m+1}^{n-m}\not\in\mathcal{B}(1)$ and $x_{m+1}^{n-m}\in\mathcal{B}((1\;m+1))$. Thus $a_{d,k}(i,I_1,\ldots,I_k)$ and the canonical basis of $V(m,n)_{2n-m+1}$ are linearly independent, as desired.
\end{proof}

\begin{cor}
  \label{p_n surjective}
  The map $p_n\colon\H(m,U(n+1))^{\Sigma_{n+1}}\to \H(m,U(n))^{\Sigma_n}$ is an isomorphism in dimension $2n-m+1$.
\end{cor}

\begin{proof}
  By Lemma \ref{dim 2n-m+1}, the map $p_n$ induces an isomorphism $V(m,n+1)_{2n-m+1}\cong V(m,n)_{2n-m+1}\oplus U$. Then the proof is done by Lemma \ref{inj}, Theorem \ref{low dim} and Proposition \ref{V+U}.
\end{proof}

We define
$$d_{m,n}\coloneqq\begin{cases}2n-m+1&(G_n=U(n),SU(n))\\2n+1&(G_n=Sp(n),SO(2n+1)).\end{cases}$$

\begin{thm}
  [Theorem \ref{homology stability intro}]\label{homology stability}
  Suppose that $n\ge m$ for $G_n=U(n),SU(n)$ and $m\ge 2$ for $G_n=Sp(n),SO(2n+1)$. Then the map
  \begin{equation}
    \label{induced map}
    H_*(\Hom(\Z^m,G_n)_1;\Q)\to H_*(\Hom(\Z^m,G_{n+1})_1;\Q)
  \end{equation}
  is an isomorphism for $*\le d_{m,n}$ and not surjective for $*=d_{m,n}+1$.
\end{thm}

\begin{proof}
  We assume that the ground field is $\Q$. Let $W_n$ denote the Weyl group of $G_n$. There is a natural projection
  $$p_n\colon\H(m,G_{n+1})^{W_{n+1}}\to\H(m,G_n)^{W_n}.$$
  By the construction of the isomorphism in Theorem \ref{Baird}, this projection satisfies a commutative square
  $$\xymatrix{H^*(\Hom(\Z^m,G_{n+1})_1;\Q)\ar[r]^\cong\ar[d]&\H(m,G_{n+1})^{W_{n+1}}\ar[d]^{p_n}\\
  H^*(\Hom(\Z^m,G_n)_1;\Q)\ar[r]^\cong&\H(m,G_n)^{W_n}.}$$
  Then it suffices to show that $p_n$ is an isomorphism in dimension $\le d_{m,n}$ and not injective in dimension $d_{m,n}+1$.

  First, we consider the case $G_n=Sp(n)$. Clearly, $p_n(w(d,I))=w(d,I)$ for $2d+|I|+\epsilon(|I|)-2\le 2n$. Then it follows from Theorem \ref{low dim Sp} that $p_n$ is an isomorphism in dimension $\le d_{m,n}$. Since we are supposing $m\ge 2$, there is an element
  $$w(1,\{1,2\})=y_1^1y_1^2+\cdots+y_k^1y_k^2\in\H(m,Sp(k))^{B_k}$$
  for $k=n,n+1$. It is obvious that $w(1,\{1,2\})^{n+1}=0$ in $\H(m,Sp(n))^{B_n}$, and by Theorem \ref{low dim Sp}, $w(1,\{1,2\})^{n+1}\ne0$ in $\H(m,Sp(n+1))^{B_{n+1}}$. Thus the statement is proved. By Theorem \ref{cohomology SO Sp}, the case $G_n=SO(2n+1)$ is also proved.

  Next, we consider the case $G_n=U(n)$. By definition, $p_n(z(d,I))=z(d,I)$ for $d+|I|=1\le n$. Then by Theorem \ref{low dim}, $p_n$ is an isomorphism in dimension $\le d_{m,n}-1$. The element $z(n-m+2,[m])$ belongs to $\S(m,U(n+1))^{\Sigma_{n+1}}$, but it does not belong to $\S(m,U(n))^{\Sigma_n}$. Since $z(n-m+2,[m])$ is of degree $2n-m+2$, this implies that $p_n$ is not injective in dimension $d_{m,n}+1$. Then it remains to show that $p_n$ is an isomorphism in dimension $d_{m,n}$, and this is already proved by Corollary \ref{p_n surjective}. Thus the proof for $G_n=U(n)$ is complete. The case $G_n=SU(n)$ follows from this result and \eqref{U-SU}.
\end{proof}

\begin{rem}
  \begin{enumerate}
    \item Let $m=1$ and $G_n=Sp(n),SO(2n+1)$. Since $\Hom(\Z,G_n)_1=G_n$, the map \eqref{induced map} is an isomorphism for $*\le 4n+2$ and not surjective for $*=4n+3$.

    \item Let $n<m$. Then quite similarly to the above proof, we can show that the map \eqref{induced map} is an isomorphism for $*\le n$ and not surjective for $*=n+1$.
  \end{enumerate}
\end{rem}


\section{Questions}\label{Questions}

In this section, we pose several questions arising in our work.

In Section \ref{Poincare series}, we gave explicit formulae for the Poincar\'e series of $\Hom(\Z^m,G)_1$ by investigating combinatorics of (signed) permutation. So the method does not directly apply to the exceptional Lie groups. On the other hand, by a direct calculation of the formula in Theorem \ref{Ramras-Stafa} using an explicit matrices for the canonical representation of the Weyl group of $G_2$, Ramras and Stafa \cite{RS1} gave an explicit formula for the Poincar\'e series of $\Hom(\Z^m,G_2)_1$. However, this formula does not show a connection to combinatorics of the Weyl group. Then we ask the following question.

\begin{quest}
  Is there a formula for the Poincar\'e series of $\Hom(\Z^m,G)_1$ in terms of combinatorics of the Weyl group of $G$ when $G$ is exceptional?
\end{quest}

A formula for the Poincar\'e series of $\Hom(\Z^m,G)_1$ when $G$ is exceptional must be more computable than the one in Theorem \ref{Ramras-Stafa}. We determined the top term of the Poincar\'e series by using our formula for the classical groups.

\begin{quest}
  Can we determine the top term of the Poincar\'e series of $\Hom(\Z^m,G)_1$ when $G$ is exceptional?
\end{quest}

As we can see in Theorem \ref{top term general}, the top term of the Poincar\'e series of $\Hom(\Z^2,G)_1$ is given in terms of simple factors of $G$. This was obtained by a formal calculation using our formula and the computer calculation in Appendix. Then we do not understand its topological meaning.

\begin{quest}
  Is there a topological interpretation of the identity in Theorem \ref{top term general}?
\end{quest}

In Section \ref{Cohomology generators}, we gave a minimal generating set of $\H(m,G)^W$ for the classical group $G$, where there is an isomorphism
$$H^*(\Hom(\Z^m,G)_1;\F)\cong\H(m,G)^W$$
for a field $\F$ of characteristic zero or prime to $|W|$. Our method is a direct calculation of invariants of (signed) symmetric groups, which does not apply to other Lie groups. However, for $m=1$ we have the theorem of Solomon, from which we can determine $\H(1,G)^W$. By the Shephard-Todd theorem, we have
$$\P(n)^W=\F[f_1,\ldots,f_n]$$
where $G$ is of rank $n$. Define a map $d\colon\P(n)\to\P(n)\otimes\Lambda(y_1,\ldots,y_n)$ by $dx_i=y_i$ together with the Leibniz rule. Then the theorem of Solomon states that
$$\widetilde{\H}(1,G)^W\cong\P(n)^W\otimes\Lambda(df_1,\ldots,df_n).$$
As in the proof of Theorem \ref{Ramras-Stafa},
$$\H(1,G)^W=\widetilde{\H}(1,G)^W/(f_1,\ldots,f_n).$$
Then we obtain
$$\H(1,G)^W=\Lambda(df_1,\ldots,df_n)$$
which is a well known result in topology. Applying this observation to \eqref{power sum}, we obtain Theorem \ref{minimal generator U}. Note that our generators $z(d,I)$ can be obtained by "differentiating" the power sums $p_1,\ldots,p_n$. Then Theorems \ref{minimal generator U}, \ref{minimal generator Sp} and Corollary \ref{minimal generator SU}, \ref{minimal generator SO} may be alternatively proved by a representation theoretic way.

\begin{quest}\label{Solomon}
  Is there a generalization of the theorem of Solomon, which gives a minimal generating set of $\H(m,G)^W$?
\end{quest}

By the theorem of Solomon, if Lie groups $G_1$ and $G_2$ are simple of the same rank, then the cohomology of $\Hom(\Z,G_1)_1$ and $\Hom(\Z,G_2)$ are isomorphic as ungraded rings. In \cite{T}, the second author determines the cohomology of $\Hom(\Z^2,G)_1$ when $G$ is simple Lie group of rank two, including the exceptional group $G_2$. In particular, he finds that cohomology of these spaces are isomorphic to each other as ungraded rings. Considering the case of $\Hom(\Z,G)_1$ above, this supports the affirmative answer to this question.

As we can see in Examples \ref{SU(2)} and \ref{SU(3)}, we do not have a general scheme to get relations among our minimal generating set of $\H(m,G)^W$.

\begin{quest}
  Is there a general scheme to get relations among our minimal generating set of $\H(m,G)^W$?
\end{quest}

The best answer to this question is an answer to Question \ref{Solomon}, which completely describes $\H(m,G)^W$ in terms of $W$.


\appendix
\section{Exceptional Lie groups}\label{Exceptional Lie groups}

We can calculate the Poincar\'e series of $\Hom(\Z^2,G)_1$ for the exceptional Lie group $G$ by Theorem \ref{Ramras-Stafa} with an assistance of computer. In this section, we only record the result.

\begin{align*}
  &P(\Hom(\Z^2,G_2)_1;t)=1+t^2+2t^3+t^4+2t^5+t^6+t^{10}+2t^{11}+2t^{13}+3t^{14}\\
  &P(\Hom(\Z^2,F_4)_1;t)=1+t^2+2 t^3+t^4+2 t^5+2 t^6+2 t^7+2 t^8+2 t^9+2 t^{10}+2 t^{11}\\
  &\quad+t^{12}+4 t^{13}+6 t^{14}+6 t^{15}+6 t^{16}+8 t^{17}+9
  t^{18}+6 t^{19}+6 t^{20}+6 t^{21}+6 t^{22}+6 t^{23}\\
  &\quad+5 t^{24}+8 t^{25}+9 t^{26}+8 t^{27}+14 t^{28}+12 t^{29}+8 t^{30}+10 t^{31}+7 t^{32}+4 t^{33}+4
  t^{34}\\
  &\quad+4 t^{35}+8 t^{36}+8 t^{37}+6 t^{38}+8 t^{39}+10 t^{40}+6 t^{41}+3 t^{42}+6 t^{43}+3 t^{44}+3 t^{48}\\
  &\quad+4 t^{49}+4 t^{51}+5 t^{52}\\
  &P(\Hom(\Z^2,E_6)_1;t)=1+t^2+2 t^3+t^4+2 t^5+2 t^6+2 t^7+3 t^8+4 t^9+4 t^{10}+8 t^{11}\\
  &\quad+8 t^{12}+10 t^{13}+13 t^{14}+12 t^{15}+15 t^{16}+20 t^{17}+20t^{18}+26 t^{19}+32 t^{20}+30 t^{21}\\
  &\quad+39 t^{22}+44 t^{23}+44 t^{24}+54 t^{25}+62 t^{26}+62 t^{27}+73 t^{28}+76 t^{29}+82 t^{30}+92 t^{31}\\
  &\quad+92 t^{32}+100t^{33}+110 t^{34}+106 t^{35}+109 t^{36}+120 t^{37}+118 t^{38}+126 t^{39}+129 t^{40}\\
  &\quad+124 t^{41}+132 t^{42}+128 t^{43}+115 t^{44}+128 t^{45}+126 t^{46}+116 t^{47}+118 t^{48}+108 t^{49}\\
  &\quad+101 t^{50}+104 t^{51}+89 t^{52}+88 t^{53}+99 t^{54}+78 t^{55}+68 t^{56}+70 t^{57}+55 t^{58}+48 t^{59}\\
  &\quad+51 t^{60}+40 t^{61}+44 t^{62}+40 t^{63}+21 t^{64}+24 t^{65}+26 t^{66}+14 t^{67}+18 t^{68}+20 t^{69}\\
  &\quad+10 t^{70}+10 t^{71}+5 t^{72}+5 t^{74}+6 t^{75}+6 t^{77}+7 t^{78}\\
  &P(\Hom(\Z^2,E_7)_1;t)=1+t^2+2 t^3+t^4+2 t^5+2 t^6+2 t^7+2 t^8+2 t^9+3 t^{10}+4 t^{11}\\
  &\quad+3 t^{12}+6 t^{13}+8 t^{14}+8 t^{15}+8 t^{16}+10 t^{17}+12 t^{18}+12 t^{19}+14 t^{20}+18 t^{21}+19 t^{22}\\
  &\quad+20 t^{23}+24 t^{24}+28 t^{25}+29 t^{26}+30 t^{27}+39 t^{28}+44 t^{29}+44 t^{30}+48 t^{31}+58 t^{32}\\
  &\quad+64 t^{33}+66 t^{34}+70 t^{35}+82 t^{36}+90 t^{37}+91 t^{38}+98 t^{39}+112 t^{40}+118 t^{41}+123 t^{42}\\
  &\quad+134 t^{43}+146 t^{44}+152 t^{45}+158 t^{46}+174 t^{47}+183 t^{48}+186 t^{49}+198 t^{50}+214 t^{51}\\
  &\quad+221 t^{52}+222 t^{53}+235 t^{54}+256 t^{55}+252 t^{56}+252 t^{57}+274 t^{58}+288 t^{59}+282 t^{60}\\
  &\quad+282 t^{61}+300 t^{62}+314 t^{63}+303 t^{64}+300 t^{65}+319 t^{66}+328 t^{67}+313 t^{68}+306 t^{69}\\
  &\quad+326 t^{70}+330 t^{71}+313 t^{72}+314 t^{73}+324 t^{74}+320 t^{75}+305 t^{76}+296 t^{77}+300 t^{78}\\
  &\quad+292 t^{79}+277 t^{80}+278 t^{81}+278 t^{82}+260 t^{83}+256 t^{84}+256 t^{85}+241 t^{86}+222 t^{87}\\
  &\quad+219 t^{88}+220t^{89}+200 t^{90}+178 t^{91}+181 t^{92}+180 t^{93}+160 t^{94}+140 t^{95}+145 t^{96}\\
  &\quad+158 t^{97}+129 t^{98}+104 t^{99}+118 t^{100}+112 t^{101}+84 t^{102}+70t^{103}+76 t^{104}+88 t^{105}\\
  &\quad+68 t^{106}+50 t^{107}+63 t^{108}+66 t^{109}+43 t^{110}+30 t^{111}+38 t^{112}+42 t^{113}+25 t^{114}\\
  &\quad+16 t^{115}+25 t^{116}+26t^{117}+17 t^{118}+12 t^{119}+17 t^{120}+22 t^{121}+12 t^{122}+6 t^{123}\\
  &\quad+12 t^{124}+6 t^{125}+6 t^{129}+7 t^{130}+7 t^{132}+8 t^{133}
\end{align*}

\begin{align*}
  &P(\Hom(\Z^2,E_8)_1;t)=1+t^2+2 t^3+t^4+2 t^5+2 t^6+2 t^7+2 t^8+2 t^9+2 t^{10}+2 t^{11}\\
  &\quad+2 t^{12}+2 t^{13}+3 t^{14}+4 t^{15}+3 t^{16}+6 t^{17}+6t^{18}+4 t^{19}+6 t^{20}+6 t^{21}+7 t^{22}+8 t^{23}\\
  &\quad+7 t^{24}+10 t^{25}+12 t^{26}+12 t^{27}+13 t^{28}+16 t^{29}+17 t^{30}+16 t^{31}+18 t^{32}+18 t^{33}\\
  &\quad+20t^{34}+22 t^{35}+21 t^{36}+60 t^{48}+64t^{49}+64 t^{50}+64 t^{51}+75 t^{52}+74 t^{53}+76 t^{54}\\
  &\quad+90 t^{55}+97 t^{56}+100 t^{57}+107 t^{58}+116 t^{59}+122 t^{60}+126 t^{61}+124 t^{62}+128t^{63}\\
  &\quad+145 t^{64}+144 t^{65}+151 t^{66}+180 t^{67}+191 t^{68}+198 t^{69}+210 t^{70}+218 t^{71}+225 t^{72}\\
  &\quad+224 t^{73}+218 t^{74}+228 t^{75}+245 t^{76}+242t^{77}+258 t^{78}+298 t^{79}+310 t^{80}+328 t^{81}\\
  &\quad+345 t^{82}+348 t^{83}+361 t^{84}+354 t^{85}+347 t^{86}+364 t^{87}+374 t^{88}+382 t^{89}+413 t^{90}\\
  &\quad+452t^{91}+464 t^{92}+488 t^{93}+509 t^{94}+494 t^{95}+493 t^{96}+486 t^{97}+478 t^{98}+494 t^{99}\\
  &\quad+496 t^{100}+520 t^{101}+570 t^{102}+598 t^{103}+615t^{104}+650 t^{105}+668 t^{106}+640 t^{107}\\
  &\quad+625 t^{108}+616 t^{109}+607 t^{110}+614 t^{111}+606 t^{112}+646 t^{113}+707 t^{114}+708 t^{115}\\
  &\quad+721 t^{116}+758t^{117}+760 t^{118}+720 t^{119}+687 t^{120}+676 t^{121}+671 t^{122}+670 t^{123}\\
  &\quad+663 t^{124}+716 t^{125}+782 t^{126}+768 t^{127}+785 t^{128}+810 t^{129}+784t^{130}+744 t^{131}\\
  &\quad+693 t^{132}+672 t^{133}+668 t^{134}+650 t^{135}+652 t^{136}+704 t^{137}+753 t^{138}+736 t^{139}\\
  &\quad+745 t^{140}+756 t^{141}+710 t^{142}+678t^{143}+627 t^{144}+594 t^{145}+600 t^{146}+578 t^{147}\\
  &\quad+591 t^{148}+640 t^{149}+658 t^{150}+652 t^{151}+661 t^{152}+638 t^{153}+571 t^{154}+546 t^{155}\\
  &\quad+504t^{156}+456 t^{157}+460 t^{158}+448 t^{159}+468 t^{160}+506 t^{161}+498 t^{162}+512 t^{163}\\
  &\quad+521 t^{164}+474 t^{165}+417 t^{166}+400 t^{167}+369 t^{168}+322t^{169}+323 t^{170}+322 t^{171}\\
  &\quad+348 t^{172}+372 t^{173}+334 t^{174}+360 t^{175}+376 t^{176}+302 t^{177}+254 t^{178}+248 t^{179}\\
  &\quad+226 t^{180}+190 t^{181}+186t^{182}+202 t^{183}+228 t^{184}+240 t^{185}+208 t^{186}+230 t^{187}\\
  &\quad+247 t^{188}+178 t^{189}+147 t^{190}+146 t^{191}+124 t^{192}+104 t^{193}+93 t^{194}+110t^{195}\\
  &\quad+136 t^{196}+132 t^{197}+104 t^{198}+122 t^{199}+140 t^{200}+82 t^{201}+62 t^{202}+72 t^{203}\\
  &\quad+53 t^{204}+48 t^{205}+42 t^{206}+54 t^{207}+83t^{208}+74 t^{209}+56 t^{210}+68 t^{211}+71 t^{212}\\
  &\quad+36 t^{213}+24 t^{214}+30 t^{215}+17 t^{216}+16 t^{217}+12 t^{218}+16 t^{219}+39 t^{220}+30 t^{221}\\
  &\quad+19t^{222}+30 t^{223}+31 t^{224}+14 t^{225}+7 t^{226}+14 t^{227}+7 t^{228}+6 t^{229}+7 t^{230}\\
  &\quad+6 t^{231}+19 t^{232}+14 t^{233}+7 t^{234}+14 t^{235}+7t^{236}+7 t^{244}+8 t^{245}+8 t^{247}+9 t^{248}
\end{align*}



\begin{thebibliography}{99}
  \bibitem{AC} A. Adem and F.R. Cohen, \emph{Commuting elements and spaces of homomorphisms}, Math. Ann. \textbf{338} (2007) 587-626.

  \bibitem{ACG} A. Adem, F.R. Cohen, and J.M. G\'omez, \emph{Stable splittings, spaces of representations and almost commuting elements in Lie groups}, Math. Proc. Camb. Phil. Soc. \textbf{149} (2010) 455-490.

  \bibitem{A} V.I. Arnol’d, \emph{The cohomology ring of the colored braid group}, Mathematical Notes. \textbf{5} (1969), 138-140.

  \bibitem{B} T.J. Baird, \emph{Cohomology of the space of commuting $n$-tuples in a compact Lie group}, Algebr. Geom. Topol. \textbf{7} (2007), 737-754.

  \bibitem{BJS} T. Baird, L.C. Jeffrey, and P. Selick, \emph{The space of commuting $n$-tuples in $SU(2)$}, Illinois J. Math. \textbf{55} (2011), 805-813.

  \bibitem{BFM} A. Borel, R. Friedman, and J. Morgan, \emph{Almost commuting elements in compact Lie groups}, Mem. Amer. Math. Soc. \textbf{157} (2002).

  \bibitem{CE} T. Church and J.S. Ellenberg, \emph{Homology of $\mathrm{FI}$-modules}, Geom. Topol. \textbf{21} (2017), 2373-2418.

  \bibitem{CF} T. Church and B. Farb, \emph{Representation theory and homological stability}, Adv. in Math. \textbf{245} (2013), 250-314.

  \bibitem{CS1} F.R. Cohen and M. Stafa, \emph{A survey on spaces of homomorphisms to Lie groups}, Configurations Spaces: Geometry, Topology and Representation Theory, Springer INdAM series \textbf{14} (2016), 361-379.

  \bibitem{CS2} F.R. Cohen and M. Stafa, \emph{On spaces of commuting elements in Lie groups}, Math. Proc. Camb. Philos. Soc. \textbf{161} (2016), 381-407.

  \bibitem{C} M.C. Crabb, \emph{Spaces of commuting elements in $SU(2)$}, Proc. Edinburgh Math. Soc. \textbf{54} (2011) 67-75.

  \bibitem{FHT} Y. F\'elix, S. Halperin, and J.-C. Thomas, \emph{Rational Homotopy Theory}, Graduate texts in Math. \textbf{205}, Springer-Verlag, New York, 2001.

  \bibitem{F} W. Fulton, \emph{Young tableaux. With applications to representation theory and geometry}, London Math. Soc. Student Texts, \textbf{35}, Cambridge University Press, Cambridge, 1997.

  \bibitem{GKRW} S. Galatius, A. Kupers, and O. Randal-Williams, \emph{Cellular $E_k$-algebras}, \url{arxiv.org/abs/1805.07184}.

  \bibitem{GRW} S. Galatius and O. Randal-Williams, \emph{Homological stability for moduli spaces of high dimensional manifolds. I}, J. Amer. Math. Soc. \textbf{31} (2018), 215-264.

  \bibitem{GPS} J. Gome\'z, A. Pettet, and J. Souto, \emph{On the fundamental group of $\Hom(\Z^k,G)$}, Math. Z. \textbf{271} (2012), 33-44.

  \bibitem{Ha} J.L. Harer, \emph{Stability of the homology of the mapping class groups of orientable surfaces}, Ann. of Math. \textbf{121} (1985), 215-249.

  \bibitem{He} R. Hepworth, \emph{Homological stability for families of Coxeter groups}, Algebr. Geom. Topol. \textbf{16} (2016), 2779-2811.

  \bibitem{KS} V. Kac and A. Smilga, \emph{Vacuum structure in supersymmetric Yang-Mills theories with any gauge group}, The many faces of the superworld (2000), 185-234, World Sci. Publishing, River Edge, NJ, 2000.

  \bibitem{K} M. Krannich, \emph{Homological stability of topological moduli spaces}, Geom. Topol. \textbf{23} (2019), no. 5, 2397-2474.

  \bibitem{N} M. Nakaoka, \emph{Decomposition theorem for homology groups of symmetric groups}, Ann. of Math. (2) \textbf{71} (1960), 16-42.

  \bibitem{PS} A. Pettet and J. Souto, \emph{Commuting tuples in reductive groups and their maximal compact subgroups}, Geom. Topol. \textbf{17} (2013), 2513-2593.

  \bibitem{Q} D. Quillen, \emph{Finite generation of the groups $K_i$ of rings of algebraic integers}, Algebraic $K$-theory, I: Higher $K$-theories, Lecture Notes in Math. \textbf{341}, 179–198, Springer, Berlin, 1973.

  \bibitem{RS1} D.A. Ramras and M. Stafa, \emph{Hilbert-Poincar\'e series for spaces of commuting elements in Lie groups}, Math. Z. \textbf{292} (2019), 591-610.

  \bibitem{RS2} D.A. Ramras and M. Stafa, \emph{Homological stability for spaces of commuting elements in Lie groups}, \url{arxiv.org/abs/1805.01368}.

  \bibitem{STG} D. Sjerveand E. Torres-Giese, \emph{Fundamental groups of commuting elements in Lie groups}, Bull. Lond. Math. Soc. \textbf{40} (2008), 65-76.

  \bibitem{S} R.P. Stanley, \emph{Enumerative Combinatorics: Volume 1}, Cambridge Studies in Advanced Math. \textbf{49} 2nd Edition, Cambridge University Press, 2012.
  
  \bibitem{T}M. Takeda, \emph{Cohomology of the spaces of commuting elements in Lie groups of rank two}, in preparation.

  \bibitem{W1} E. Witten, \emph{Constraints on supersymmetry breaking}, Nuclear Phys. B \textbf{202} (1982), 253-316.

  \bibitem{W2} E. Witten, \emph{Toroidal compactification without vector structure}, J. High Energy Phys. Paper 6 (1998).
\end{thebibliography}
\end{document}